\documentclass[11pt,a4paper]{article}
\usepackage[utf8]{inputenc}
\usepackage[english]{babel}
\usepackage{xcolor}

\usepackage{amssymb,mathtools,amsmath,amsthm,stmaryrd}
\usepackage[all,cmtip,2cell]{xy} 

\usepackage{breakurl}
\usepackage[draft=false, hidelinks, linkcolor = blue, breaklinks]{hyperref}
\usepackage{cleveref}
\usepackage{tensor}
\usepackage{url}
\usepackage{bookmark}

\usepackage{graphicx}
\usepackage{lmodern}

\usepackage[left=2cm,right=2cm,top=2.5cm,bottom=2.5cm]{geometry}

\def\defthm#1#2#3#4{
  \newtheorem{#1}[theorem]{#3}
  \newtheorem*{#1*}{#3}
  \newtheorem{#2}[theorem]{#4}
  \newtheorem*{#2*}{#4}
  \crefname{#1}{#3}{#4}
  \crefname{#2}{#4}{#4}  
}

\newtheoremstyle{mythm}%
{10pt}
{}
{\itshape}
{}
{\bf}
{.}
{.5em}
{}%

\newtheoremstyle{mydef}%
{10pt}
{3pt}
{}
{}
{\bf}
{.}
{.5em}
{}%

\newtheoremstyle{myrmk}%
{10pt}
{3pt}
{}
{}
{\bf}
{.}
{.5em}
{}%

\theoremstyle{mythm}
\newtheorem{theorem}{Theorem}[section]
\newtheorem*{theorem*}{Theorem}

\defthm{corollary}{corollaries}{Corollary}{Corollaries}
\defthm{lemma}{lemmata}{Lemma}{Lemmata}
\defthm{proposition}{propositions}{Proposition}{Propositions}
\defthm{axiom}{axioms}{Axiom}{Axioms}
\defthm{propdef}{props/defs}{Proposition/Definition}{Prositions/Definitions}
\defthm{defcor}{defscors}{Corollary/Definition}{Corollaries/Definitions}
\defthm{conjecture}{conjectures}{Conjecture}{Conjectures}

\theoremstyle{mydef}
\defthm{definition}{definitions}{Definition}{Definitions}

\theoremstyle{myrmk}
\defthm{acknowledgment}{acknowledgments}{Acknowledgment}{Acknowledgments}
\defthm{remark}{remarks}{Remark}{Remarks}
\defthm{motivation}{motivations}{Motivation}{Motivations}
\defthm{example}{examples}{Example}{Examples}
\defthm{question}{questions}{Question}{Questions}
\defthm{observation}{observations}{Observation}{Observations}
\defthm{claim}{claims}{Claim}{Claims}
\defthm{notation}{notations}{Notation}{Notations}

\newtheorem*{replemmax}{\reptitle}
 {\end{replemmax}}

\newtheorem*{repthmx}{\reptitle}
 {\end{repthmx}}

\newtheorem*{repcorx}{\reptitle}
 {\end{repcorx}}

\crefname{section}{Section}{Sections}
\crefname{theorem}{Theorem}{Theorems}

\makeatletter
\renewenvironment{proof}[1][\proofname] {\par\pushQED{\qed}\normalfont\topsep6\p@\@plus6\p@\relax\trivlist\item[\hskip\labelsep\bf#1\@addpunct{.}]\ignorespaces}{\popQED\endtrivlist\@endpefalse}
\makeatother

\newcommand{\blank}{\mbox{\hspace{3pt}\underline{\ \ }\hspace{2pt}}}
\newcommand{\sprime}{^{\prime}}

\newcommand{\pbs}{\scalebox{1.5}{\rlap{$\cdot$}$\lrcorner$}}

\newcommand{\Address}{{
  \bigskip
  \footnotesize

\textsc{Department of Mathematics and Statistics, Masaryk University, Kotl\'{a}\v{r}sk\'{a} 2, Brno 61137, Czech Republic}\par\nopagebreak
  \textit{E-mail address}: \texttt{stenzelr@math.muni.cz}

}}

\author{Raffael Stenzel}

\title{On notions of compactness, object classifiers and weak Tarski universes}

\begin{document}
\maketitle
\abstract{
We prove a correspondence between $\kappa$-small fibrations in simplicial presheaf categories equipped 
with the injective or projective model structure (and left Bousfield localizations thereof) and 
relatively $\kappa$-compact maps in their underlying quasi-categories for suitably large regular cardinals $\kappa$. 
We thus obtain a transition result between weakly universal small fibrations in the (type theoretic) injective Dugger-Rezk-style 
standard presentations of model toposes and object classifiers in Grothendieck $\infty$-toposes in the sense of Lurie.
}

\section{Introduction}
A Grothendieck $\infty$-topos $\mathcal{M}$ is an accessible left exact localization of the presheaf $(\infty,1)$-category
$\hat{\mathcal{C}}$ over some small $(\infty,1)$-category $\mathcal{C}$ (\cite[Section 6.1]{luriehtt}). Hence, it is presented by a 
model topos $\mathbb{M}$ given by a left exact left Bousfield localization of the simplicial category $\mathbf{sPsh}(\mathbf{C})$ of 
simplicially enriched presheaves on an associated small simplicial category $\mathbf{C}$, equipped with 
either the projective or the injective model structure (\cite[Section 6]{rezkhtytps}). In the following we prove a 
correspondence between $\kappa$-small fibrations in such model toposes $\mathbb{M}$, and relatively $\kappa$-compact maps in their 
associated Grothendieck $\infty$-toposes $\mathcal{M}=\mathrm{Ho}_{\infty}(\mathbb{M})$ for regular cardinals $\kappa$ large enough. This is 
motivated by the interpretation of univalent Tarski universes defined in Martin-L\"{o}f type theory (\cite{hott}) as univalent fibrations 
universal for the class of $\kappa$-small fibrations for suitable cardinals $\kappa$ (\cite{cisinski}, \cite{shulmanuniverses}), and their 
intended interpretation as object classifiers in higher topos theory for relatively $\kappa$-compact maps as developed in
\cite[Section 6.1.6]{luriehtt}. Therefore, even though we prove an analogous (but slightly weaker result) result for the projective model 
structure and arbitrary localizations, the main result of this paper is the following.

\setcounter{section}{3}
\setcounter{theorem}{20}
\begin{theorem}
Let $\mathbf{C}$ be a small simplicial category, $T$ be a set of arrows in the underlying category of $\mathbf{sPsh}(\mathbf{C})$ and
$\mathbb{M}$ be the left Bousfield localization $\mathcal{L}_T\mathbf{sPsh}(\mathbf{C})_{\mathrm{inj}}$ of
$\mathbf{sPsh}(\mathbf{C})$ equipped with the injective model structure at the set $T$ (in the classic sense of
\cite[Section 3.3]{hirschhorn03}). Assume that the localization is left exact, and let $\kappa$ be a sufficiently large 
regular cardinal. Then a morphism $f\in\mathrm{Ho}_{\infty}(\mathbb{M})$ is relatively $\kappa$-compact if and only if there is a
$\kappa$-small fibration $g\in\mathrm{sPsh}(\mathbf{C})$ such that $g\simeq f$ in $\mathrm{Ho}_{\infty}(\mathbb{M})$.
\end{theorem}

\setcounter{section}{4}
\setcounter{theorem}{0}

\begin{corollary}
Let $\mathbb{M}=\mathcal{L}_T\mathbf{sPsh}(\mathbf{C})_{\mathrm{inj}}$ be a left exact left Bousfield localization as in 
Theorem~\ref{thmcmptinj}, and let $\kappa$ be a sufficiently large regular cardinal. Then a relatively $\kappa$-compact map
$p\in\mathrm{Ho}_{\infty}(\mathbb{M})$ is a 
classifying map for all relatively $\kappa$-compact maps in $\mathrm{Ho}_{\infty}(\mathbb{M})$ if and only if there is a univalent
$\kappa$-small fibration $\pi\in\mathbb{M}$ which is weakly universal for all $\kappa$-small fibrations in $\mathbb{M}$ such that 
$p\simeq\pi$ in $\mathrm{Ho}_{\infty}(\mathbb{M})$. 
\end{corollary}

As a prerequisite for the proof we are giving, in Section~\ref{sectransferskeleton} we show that up to DK-equivalence 
every simplicial category can be replaced by the localization of a well founded poset as already observed by Shulman in 
\cite{shulmannote}. This allows us to replace simplicial presheaf categories over arbitrary small simplicial categories
$\mathbf{C}$ by localizations of simplicial presheaf categories over well founded posets $I$. In Section~\ref{secsmallvscompact}, we 
will use that such model categories come equipped with a theory of minimal fibrations that will allow us to present relatively
$\kappa$-compact maps in their underlying quasi-category by $\kappa$-small fibrations. Those can be pushed forward to
$\kappa$-small projective fibrations in our original presheaf category over $\mathbf{C}$ making use of Dugger's ideas about universal 
homotopy theories in \cite{duggerunivhtytheories}. The move to the injective model structure then follows by Shulman's 
recent observation (\cite[Section 8]{shulmanuniverses}) that the cobar construction on simplicial presheaf categories takes projective 
fibrations to injective ones.
In Section~\ref{secuniverses} we explain the relevance of this result for the semantics of Homotopy Type Theory in 
higher topos theory, as Theorem~\ref{thmcmptinj} is necessary to translate Tarski universes in the syntax to object 
classifiers in an $\infty$-topos (when using the common semantics via type theoretic model categories given in
\cite[Section 4]{shulmaninv}).
	
\paragraph{Acknowledgments}
The author would like to thank his advisor Nicola Gambino at the University of Leeds for the frequent discussions 
and references, as well as his feedback on all stages of the writing process. The author is 
also grateful to Karol Szumi\l{}o and Mike Shulman for helpful discussions and, especially, to Mike Shulman for sharing 
his note on the presentation of small $(\infty,1)$-categories as localizations of inverse posets, which provided a 
crucial step to the proof of Theorem~\ref{thmcmptinj}. Also thanks to Jarl Taxer\aa s Flaten who realized that the assumption of a 
Mahlo cardinal in an earlier version of this paper is unnecessary (see Remark~\ref{remlargecardinals}).

Most of the work for this paper was carried out as part of the author's PhD thesis, supported by a Faculty Award of the University of 
Leeds 110 Anniversary Research Scholarship. It was finished with support of the Grant Agency of the Czech Republic under the grant 
22-02964S.
\setcounter{section}{1}
\setcounter{theorem}{0}

\section{Direct poset presentations of simplicial categories}\label{sectransferskeleton}

In the following, simplicial categories - that is simplicially enriched categories - will be denoted by bold faced 
letters $\mathbf{C}$ and ordinary categories will be distinguished by blackboard letters $\mathbb{C}$.
$\mathbf{S}$ denotes the (simplicial) category of simplicial sets. By a simplicial presheaf over $\mathbf{C}$ we mean a 
simplicially enriched presheaf $X\colon\mathbf{C}^{op}\rightarrow\mathbf{S}$. Simplicial presheaves and simplicial natural 
transformations form part of a simplicial category $\mathbf{sPsh}(\mathbf{C})$ (via the usual end-construction, denoted
$[\mathbf{C}^{op},\mathbf{S}]$ in \cite[Section 2.2]{kellybook}) whose underlying ordinary category will be denoted by
$\mathrm{sPsh(\mathbf{C})}$.

Mike Shulman noted in \cite[Lemma 0.2]{shulmannote} that every quasi-category can be presented by the 
localization of a direct -- in other words, well-founded -- poset.\footnote{Shulman in fact argues for a presentation by 
inverse posets. But since 
localization commutes with taking opposite categories, this amounts to the same statement.}
Since the note is unpublished, in this section we present a slightly stronger variation of his observation (with an accordingly 
slightly different proof) and discuss the resulting presentations of associated presheaf $(\infty,1)$-categories.
Although the following sections only will require the fact that every $(\infty,1)$-category can be 
presented by the localization of an Eilenberg-Zilber category (\cite{bmezcat}), proving the stronger condition of 
posetality only requires about as much work as the Eilenberg-Zilber condition itself.\\

Recall the following constructions and notation from \cite{bkrelcat}. A \emph{relative category} is a pair
$(\mathbb{C},V)$ such that $\mathbb{C}$ is a category and
$V$ is a subcategory of $\mathbb{C}$. A \emph{relative functor}
$F\colon(\mathbb{C},V)\rightarrow(\mathbb{D},W)$ is a functor
$F\colon\mathbb{C}\rightarrow\mathbb{D}$ of categories such that $F[V]\subseteq W$. The relative functor $F$ is a 
\emph{relative inclusion} if its underlying functor of categories is an inclusion and $V=W\cap\mathbb{C}$. The category 
of small relative categories and relative functors is denoted by RelCat.

There are two canonical inclusions of the category $\mathrm{Cat}$ of small categories into RelCat; for a category
$\mathbb{C}$ and its discrete wide subcategory $\mathbb{C}_0$, we obtain the associated minimal relative category
$\check{\mathbb{C}}:=(\mathbb{C},\mathbb{C}_0)$ and the associated maximal relative category
$\hat{\mathbb{C}}:=(\mathbb{C},\mathbb{C})$. 

In \cite[Section 4]{bkrelcat}, Barwick and Kan introduce a combinatorial sub-division operation
$\xi\colon\mathrm{RelPos}\rightarrow\mathrm{RelPos}$ on relative posets (considered as posetal relative categories) and an associated 
bisimplicial nerve construction $N_{\xi}\colon\mathrm{RelCat}\rightarrow s\mathbf{S}$ giving rise to the adjoint pair
\begin{align}\label{equmodeltransf}
\xymatrix{
s\mathbf{S}\ar@<.5ex>[r]^(.4){K_{\xi}} & \mathrm{RelCat}.\ar@<.5ex>[l]^(.6){N_{\xi}}
}
\end{align}
The left adjoint $K_{\xi}$ is given by $K_{\xi}(\Delta[m,n])=\xi(\check{[m]}\times\hat{[n]})$ on representables and left Kan 
extension along the Yoneda embedding. The authors of \cite{bkrelcat} have shown that the category $\mathrm{RelCat}$ 
inherits a transferred model structure from the Reedy model structure
$(s\mathbf{S},R_v)$ which turns the pair $(K_{\xi},N_{\xi})$ into a Quillen-equivalence. By construction, the set 
$K_{\xi}[\mathcal{I}_v]$ forms a set of generating cofibrations for the model structure in question, where
\[\mathcal{I}_v:=\{\partial\Delta[n,m]\hookrightarrow\Delta[n,m]\mid n,m\geq 0\}\]
denotes the generating set of monomorphisms in $s\mathbf{S}$.
Left Bousfield localization of both sides of (\ref{equmodeltransf}) induces a model structure $(\mathrm{RelCat},\mathrm{BK})$ such 
that $(K_{\xi},N_{\xi})$ is a Quillen-equivalence to Rezk's model structure $(s\mathbf{S},\mathrm{CS})$ for complete Segal spaces.

It follows that the underlying quasi-category of the model category $\mathrm{RelCat}$ is equivalent to the
quasi-category of small $(\infty,1)$-categories (\cite[Definition 1.3.4.15]{lurieha}). We will denote the thus associated small 
$(\infty,1)$-category to a relative category $(\mathbb{C},V)$ by $\mathrm{Ho}_{\infty}(\mathbb{C},V)$ whenever it is not necessary to 
specify a specific model of $(\infty,1)$-category theory.

A central notion of \cite{bkrelcat} is that of a ``Dwyer map'' in RelCat. A relative functor
$F\colon(\mathbb{C},V)\rightarrow(\mathbb{D},W)$ is a \emph{Dwyer inclusion} if $F$ is a relative 
inclusion such that $\mathbb{C}$ is a sieve in
$\mathbb{D}$ and such that the cosieve $Z\mathbb{C}$ generated by $\mathbb{C}$ in $\mathbb{D}$ comes equipped with a 
strong deformation retraction $Z\mathbb{C}\rightarrow\mathbb{C}$. The relative functor $F$ is a \emph{Dwyer map} if it 
factors as an isomorphism followed by a Dwyer inclusion, see \cite[Section 3.5]{bkrelcat} for more details.

A major insight of the authors was that the generating cofibrations
\begin{align}\label{bkgencof}
K_{\xi}(\partial\Delta[n,m])\hookrightarrow K_{\xi}(\Delta[n,m])
\end{align}
of the model category $(\mathrm{RelCat},\mathrm{BK})$ are Dwyer maps of (finite) relative posets (\cite[Proposition 9.5]{bkrelcat}). 
It follows that every cofibration in $(\mathrm{RelCat},\mathrm{BK})$ is a Dwyer map and that every cofibrant object is a relative 
poset (\cite[Theorem 6.1]{bkrelcat}).

\begin{proposition}\label{lemmacofobjrelcat}
The underlying category of a cofibrant object in $(\mathrm{RelCat},\mathrm{BK})$ is a direct (i.e.\ well-founded) poset.
\end{proposition}
\begin{proof}
Since the empty relative category $\emptyset$ is a relative direct poset, it suffices to show that for every cofibration
$(\mathbb{P},V)\hookrightarrow(\mathbb{Q},W)$ where $(\mathbb{P},V)$ is a relative direct poset also $(\mathbb{Q},W)$ is 
a relative direct poset. We show this by induction along the small object argument as follows.

The generating cofibrations (\ref{bkgencof}) are maps between finite relative posets and such are clearly direct. Both Dwyer maps and 
relative posets are closed under coproducts and under pushouts between relative posets by
\cite[Proposition 9.2]{bkrelcat}. Since Dwyer inclusions are inclusions of sieves, it is easy to see that both constructions 
preserve well-foundedness, too. Suppose we are given a 
transfinite composition of Dwyer maps $A_{\alpha}\rightarrow A_{\beta}$ for $\alpha<\beta\leq\lambda$ ordinals 
and $A_{\alpha}$ relative inverse posets. Then, as stated in the proof of \cite[Proposition 9.6]{bkrelcat}, the colimit $A_{\lambda}$ 
is a relative poset. Suppose $a=(a_i\mid i<\omega)$ is a descending sequence of arrows in $A_{\lambda}$ and let
$\alpha<\lambda$ such that $a_0\in A_{\alpha}$. Then the whole sequence $a$ is contained in $A_{\alpha}$, 
because the inclusion $A_{\alpha}\hookrightarrow A_{\lambda}$ is a Dwyer map by \cite[Proposition 9.3]{bkrelcat} 
and so $A_{\alpha}\subseteq A_{\lambda}$ is a sieve. Therefore, the sequence $a$ is finite.

In particular, every free cofibration $\emptyset\hookrightarrow(\mathbb{P},V)$ -- that is every transfinite composition 
of pushouts of coproducts of generating cofibrations with domain $\emptyset$ -- yields a relative direct poset $(\mathbb{P},V)$. But 
every cofibration $\emptyset\hookrightarrow(\mathbb{Q},W)$ is a retract of such, and hence every cofibrant object in 
RelCat is a relative direct poset.
\end{proof}

\begin{remark}
The same proof shows that the cofibrant objects in the Thomason model structure on $\mathrm{Cat}$ are direct posets, 
using Thomason's original observation that the cofibrant objects in the Thomason model structure are posetal in the 
first place.
\end{remark}

Let $F_{\Delta}\colon\mathrm{Cat}\rightarrow\mathbf{S}\text{-Cat}$
be the Bar construction obtained in the standard way by the monad resolution associated to the free 
category functor $F$ from the category $\mathrm{RGraph}$ of reflexive graphs to $\mathrm{Cat}$ (\cite[Section 2.5]{dksimploc}). Let
$U\colon\mathrm{Cat}\rightarrow\mathrm{RGraph}$ be the
corresponding right adjoint forgetful functor. Recall that $F_{\Delta}$ itself is not the left adjoint to the ``underlying category'' 
functor $(\cdot)_0\colon\mathbf{S}\text{-Cat}\rightarrow\mathrm{Cat}$, but instead, as often remarked in the literature, a cofibrant 
replacement thereof. Furthermore, recall from \cite[Section 4]{dksimploc} the (standard) simplicial localization functor
\[\mathcal{L}_{\Delta}\colon\mathrm{RelCat}\rightarrow\mathbf{S}\text{-Cat}\]
which takes a relative category $(\mathbb{C},V)$ to the simplicial category
given in degree $n\geq 0$ by
\[\mathcal{L}_{\Delta}(\mathbb{C},V)_n=F_{\Delta}(\mathbb{C})_n[F_{\Delta}(V)_n^{-1}].\]
The functor $\mathcal{L}_{\Delta}\colon(\mathrm{RelCat},\mathrm{BK})\rightarrow\mathbf{S}\text{-Cat}$ is part of an equivalence of 
associated homotopy theories (\cite{bksimploc}), so that $\mathrm{Ho}_{\infty}(C,V)\simeq\mathcal{L}_{\Delta}(C,V)$. Indeed, it is 
pointwise equivalent to the \emph{hammock localization} of a relative category (\cite[Section 2]{dkcalcsimplocs}), and hence presents the
$(\infty,1)$-categorical localization of $\mathbb{C}$ at $V$ (\cite[Remark 3.2, Theorem 3.8]{mazelgeernerve}).
Yet it has the benefit of being a strict enriched construction: it is the localization of the simplicial category
$F_{\Delta}(\mathbb{C})$ at the subcategory $F_{\Delta}(V)_0\subseteq F_{\Delta}(\mathbb{C})_0$ in the following sense.

\begin{lemma}\label{lemmasloc1}
For every relative category $(\mathbb{C},V)$ and every (potentially large) simplicial category $\mathbf{D}$, the canonical simplicial functor $j\colon F_{\Delta}(\mathbb{C})\rightarrow\mathcal{L}_{\Delta}(\mathbb{C},V)$ induces an isomorphism
\begin{align}\label{equjloc}
j^{\ast}\colon\mathbf{S}\text{-}\mathrm{Cat}(\mathcal{L}_{\Delta}(\mathbb{C},V),\mathbf{D})\rightarrow\mathbf{S}\text{-}\mathrm{Cat}_{F_{\Delta}(V)}(F_{\Delta}(\mathbb{C}),\mathbf{D})
\end{align}
of hom-categories. 
\end{lemma}
Here we consider $\mathbf{S}\text{-Cat}$ as a 2-category given by $\mathbf{S}$-enriched categories, $\mathbf{S}$-enriched functors and 
$\mathbf{S}$-enriched natural transformations (\cite[Section 1.2]{kellybook}). The right hand side of (\ref{equjloc}) denotes the full 
subcategory of $\mathbf{S}\text{-Cat}(F_{\Delta}(\mathbb{C}),\mathbf{D})$ spanned by those functors which take the arrows of 
$F_{\Delta}(V)_0$ to isomorphisms in $\mathbf{D}_0$. Thus, Lemma~\ref{lemmasloc1} states that
$F_{\Delta}(V)_0\subseteq F_{\Delta}(\mathbb{C})_0$ is ``$\mathbf{S}$-well localizable'' in the sense of \cite{wolffvloc}.

Lemma~\ref{lemmasloc1} and Corollary~\ref{propffres} are not strictly necessary for the results of this paper, but may be beneficial 
to motivate the simplicial localization construction.

\begin{proof}[Proof of Lemma~\ref{lemmasloc1}]
First, each $F_{\Delta}(V)_n$ is generated as a category 
by the iterated degeneracies of $F_{\Delta}(V)_0$. Thus, a simplicial functor $F_{\Delta}(\mathbb{C})\rightarrow\mathbf{D}$ takes all 
arrows in $F_{\Delta}(V)_n$ to isomorphisms in $\mathbf{D}_n$ for all $n\geq 0$ if and only if $F_0$ takes the arrows in
$F_{\Delta}(V)_0$ to isomorphisms in $\mathbf{D}_0$. Using the universal property of the ordinary categorical localization
$F_{\Delta}(\mathbb{C})_n[F_{\Delta}(V)_n^{-1}]$ at each degree $n\geq 0$, it follows that (\ref{equjloc}) is bijective on objects. 

Second, recall that $\mathbf{S}\text{-Cat}$ is cartesian closed (\cite[2.3]{kellybook}). Thus, Cordier's simplicial nerve 
construction $N_{\Delta}\colon\mathbf{S}\text{-Cat}\rightarrow\mathbf{S}$ with left adjoint $\mathfrak{C}$
(\cite[Section 1.1.5]{luriehtt}) 
yields a simplicial enrichment of $\mathbf{S}\text{-Cat}$ itself. Thus, the fact that $j^{\ast}$ is fully faithful follows from the 
same objectwise argument but applied to the simplicial category $[\mathfrak{C}(\Delta^1),\mathbf{D}]$. 
Indeed, enriched natural transformations of simplicial functors $\mathcal{L}_{\Delta}(\mathbb{C},V)\rightarrow\mathbf{D}$ stand in 1-1 
correspondence to 1-simplices in the nerve $N_{\Delta}([\mathcal{L}_{\Delta}(\mathbb{C},V),\mathbf{D}])$. Such correspond bijectively 
to simplicial functors $\mathfrak{C}(\Delta^1)\rightarrow[\mathcal{L}_{\Delta}(\mathbb{C},V),\mathbf{D}]$, where
$\mathfrak{C}(\Delta^1)$ is the simplicial category with two objects $0,1$, and
$\mathfrak{C}(\Delta^1)(0,1)\cong\mathfrak{C}(\Delta^1)(0,0)\cong\mathfrak{C}(\Delta^1)(1,1)\cong\Delta^0$ and $\mathfrak{C}(\Delta^1)(1,0)\cong\emptyset$. The latter functors in turn stand in 1-1 correspondence to simplicial functors
$\mathcal{L}_{\Delta}(\mathbb{C},V)\rightarrow[\mathfrak{C}(\Delta^1),\mathbf{D}]$ by \cite[Section 2.3]{kellybook}. Such functors stand in 1-1 correspondence via restriction along $j$ to simplicial functors of type
$F_{\Delta}(\mathbb{C})\rightarrow[\mathfrak{C}(\Delta^1),\mathbf{D}]$ which take arrows in $F_{\Delta}(V)_0$ to isomorphisms in
$[\mathfrak{C}(\Delta^1),\mathbf{D}]_0$ by the first part. These are exactly the enriched 
natural transformations in $\mathbf{S}\text{-Cat}_{F_{\Delta}(V)}(F_{\Delta}(\mathbb{C}),\mathbf{D})$.
\end{proof}

\begin{corollary}\label{propffres}
For $(\mathbb{C},V)\in\mathrm{RelCat}$ and
$j\colon F_{\Delta}(\mathbb{C})\rightarrow\mathcal{L}_{\Delta}(\mathbb{C},V)$ the associated localization 
functor of simplicial categories, the induced restriction
\begin{align}\label{equpropffres}
j^{\ast}\colon \mathbf{sPsh}(\mathcal{L}_{\Delta}(\mathbb{C},V))\rightarrow\mathbf{sPsh}(F_{\Delta}(\mathbb{C}))
\end{align}
of simplicial presheaf categories is fully faithful.
\end{corollary}

\begin{proof}
Given simplicial presheaves $X,Y\in\mathbf{sPsh}(\mathcal{L}_{\Delta}(\mathbb{C},V))$, we want to show that the induced map
\[j^{\ast}(X,Y)\colon\mathbf{sPsh}(\mathcal{L}_{\Delta}(\mathbb{C},V))(X,Y)\rightarrow\mathbf{sPsh}(F_{\Delta}(\mathbb{C}))(j^{\ast}X,j^{\ast}Y)\]
of simplicial sets is a bijection on all degrees $n\geq 0$. But for each such integer $n\geq 0$, the map $j^{\ast}(X,Y)_n$ is 
isomorphic to the restriction
\[j^{\ast}(X,Y)\colon\mathrm{sPsh}(\mathcal{L}_{\Delta}(\mathbb{C},V))(X\otimes\Delta^n,Y)\rightarrow\mathrm{sPsh}(F_{\Delta}(\mathbb{C}))(j^{\ast}(X\otimes\Delta^n),j^{\ast}Y),\]
since simplicial presheaf categories are cotensored over $\mathbf{S}$, and $j^{\ast}$ is cocontinuous. Thus, it is bijective by Lemma~\ref{lemmasloc1} applied to $\mathbf{D}=\mathbf{S}$.
\end{proof}

Hence, the map $j\colon F_{\Delta}(\mathbb{C})\rightarrow\mathcal{L}_{\Delta}(\mathbb{C},V)$ induces both a 
localization
\begin{equation*}
(j_!,j^{\ast})\colon \mathbf{sPsh}(F_{\Delta}(\mathbb{C}))\rightarrow \mathbf{sPsh}(\mathcal{L}_{\Delta}(\mathbb{C},V)) 
\end{equation*}
and a colocalization
\begin{equation*}
(j^{\ast},j_{\ast})\colon \mathbf{sPsh}(\mathcal{L}_{\Delta}(\mathbb{C},V))\rightarrow \mathbf{sPsh}(F_{\Delta}(\mathbb{C}))
\end{equation*}
between simplicial presheaf categories. Equipping both sides with either the injective or the projective model structure, the 
restriction $j^{\ast}$ becomes a left, respectively right Quillen functor (\cite[Proposition A.3.3.7]{luriehtt}). By \cite[Theorem 2.2]{dkequivs} applied to the map
\[j\colon(F_{\Delta}(\mathbb{C}),F_{\Delta}(V))\rightarrow (\mathcal{L}_{\Delta}(C,V),\mathcal{L}_{\Delta}(C,V)^{\cong})\]
of relative simplicial categories (\cite[Section 2.2]{bksimploc}), the 
restriction (\ref{equpropffres}) remains fully faithful on associated homotopy theories. More precisely, equipping both sides with the 
projective model structure, the pair $(j_!,j^{\ast})$ becomes a homotopy localization and induces a Quillen-equivalence 
\begin{equation}\label{lemmapseudosheaves}
(j_!,j^{\ast})\colon \mathbf{sPsh}(\mathcal{L}_{\Delta}(\mathbb{C},V))_{\mathrm{proj}}\rightarrow \mathcal{L}_{y[F_{\Delta}(V)_0]}\mathbf{sPsh}(F_{\Delta}(\mathbb{C}))_{\mathrm{proj}},
\end{equation}
where the right hand side denotes the according left Bousfield localization.
Dually, equipping both sides in (\ref{equpropffres}) with the injective model structure, the pair $(j^{\ast},j_{\ast})$ becomes a homotopy colocalization. 

The simplicial localization functor $\mathcal{L}_{\Delta}\colon\mathrm{RelCat}\rightarrow\mathbf{S}\text{-Cat}$ has a
homotopy-inverse, the ``delocalization'' or ``flattening''
\[\flat\colon\mathbf{S}\text{-Cat}\rightarrow\mathrm{RelCat},\]
given by the Grothendieck construction of a given simplicial category $\mathbf{C}$ considered as a simplicial diagram
$\mathbf{C}\colon\Delta^{op}\rightarrow\mathrm{Cat}$. This functor 
was introduced in \cite[Theorem 2.5]{dkequivs} and is analysed in detail in \cite{bksimploc}.\\

Now, given a simplicial category $\mathbf{C}$, consider its delocalization
$\flat(\mathbf{C})\in\mathrm{RelCat}$. Cofibrantly replacing $\flat(\mathbf{C})$ with some pair $(\mathbb{P},V)$ in RelCat yields a 
direct relative poset $(\mathbb{P},V)$ weakly equivalent -- i.e.\ Rezk-equivalent in the language of \cite{bksimploc} -- to
$\flat(\mathbf{C})$. Hence, by \cite[Theorem 1.8]{bksimploc}, the simplicial localization
$\mathcal{L}_{\Delta}(\mathbb{P},V)\in\mathbf{S}\text{-Cat}$ is DK-equivalent to the original simplicial category $\mathbf{C}$. That 
means there is a zig-zag of DK-equivalences of the form
\begin{equation}\tag{$\ast$}
\mathbf{C}\xrightarrow{f_1}\dots\xleftarrow{f_n}\mathcal{L}_{\Delta}(\mathbb{P},V).
\end{equation}
By \cite[Proposition A.3.3.8]{luriehtt} or \cite[Theorem 2.1]{dkequivs} and the sequence of maps in $(\ast)$, we obtain 
a zig-zag of simplicial Quillen-equivalences
\[\xymatrix{
\mathbf{sPsh}(\mathcal{L}_{\Delta}(\mathbb{P},V))_{\mathrm{proj}}\ar@<.5ex>[r]^(.7){(f_n)_{!}}& \dots\ar@<.5ex>[l]^(.3){f_n^{\ast}} 
\ar@<-.5ex>[r]_(.3){f_1^{\ast}} & \mathbf{sPsh}(\mathbf{C})_{\mathrm{proj}}.\ar@<-.5ex>[l]_(.7){(f_1)_!}
}\]
Further recall from \cite[Proposition 2.6]{dksimploc} that for every category $\mathbb{C}$ the canonical 
projection $\varphi\colon F_{\Delta}\mathbb{C}\rightarrow\mathbb{C}$ is a DK-equivalence of simplicial 
categories. So, to summarize, we have gathered the following chain of Quillen-equivalences. 

\begin{proposition}\label{proptransferskeletonproj}
Let $\mathbf{C}$ be a small simplicial category. Then there is a direct relative poset $(\mathbb{P},V)$ together with a 
zig-zag of DK-equivalences
\[\mathbf{C}\rightarrow\dots\leftarrow\mathcal{L}_{\Delta}(\mathbb{P},V)\]
in $\mathbf{S}\text{-Cat}$ which induces a zig-zag of simplicial Quillen-equivalences of the form
\begin{align*}
\xymatrix{
\mathcal{L}_{y[V]}\mathbf{sPsh}(\mathbb{P})\ar@<-.5ex>[r]_(.45){\varphi^{\ast}} & \mathcal{L}_{y[F_{\Delta}(V)_0]}\mathbf{sPsh}(F_{\Delta}\mathbb{P})\ar@<-.5ex>[l]_(.55){\varphi_!}\ar@<.5ex>[r]^(.55){j_!} & \mathbf{sPsh}(\mathcal{L}_{\Delta}(\mathbb{P},V))\ar@<.5ex>[l]^(.45){j^{\ast}}\ar@{=}[d] \\
 \mathbf{sPsh}(\mathbf{C})\ar@<.5ex>[r]^(.6){(f_1)_{!}}  & \dots\ar@<.5ex>[l]^(.4){f_1^{\ast}}\ar@<-.5ex>[r]_(.4){f_n^{\ast}} & \mathbf{sPsh}(\mathcal{L}_{\Delta}(\mathbb{P},V)),\ar@<-.5ex>[l]_(.6){(f_n)_!}
}
\end{align*}
where all simplicial presheaf categories are equipped with the projective model structure.
\end{proposition}
\begin{proof}
The only part left to show is that $\varphi\colon F_{\Delta}\mathbb{C}\rightarrow\mathbb{C}$ induces a Quillen-equivalence of
given left Bousfield localizations, but this follows directly from \cite[Proposition 2.6]{dksimploc} together with
\cite[Corollary 3.8]{dkequivs}.
\end{proof}

\section{Compactness in combinatorial model categories}\label{secsmallvscompact}

We start with some facts about compactness in presheaf categories. Given a small 
category $\mathbb{C}$, we denote the cardinality of $\mathbb{C}$ by
\[|\mathbb{C}|:=\sum_{C,C\sprime\in\mathbf{C}}|\mathrm{Hom}_{\mathbb{C}}(C,C\sprime)|.\]
Given a (set-valued) presheaf $X\in\widehat{\mathbb{C}}$, its cardinality is denoted by
\[|X|:=\sum_{C\in\mathbb{C}}|X(C)|.\]
Given a regular cardinal $\kappa>|\mathbb{C}|$, recall that a presheaf
$X\in\widehat{\mathbb{C}}$ is $\kappa$-small if $|X|<\kappa$, that is if all its values $X(C)$ have cardinality smaller 
than $\kappa$. A map $f\colon X\rightarrow Y$ in $\widehat{\mathbb{C}}$ is $\kappa$-small if all its pullbacks along 
maps $Z\rightarrow Y$ with $\kappa$-small domain $Z$ are $\kappa$-small presheaves.
Equivalently, $f\colon X\rightarrow Y$ is $\kappa$-small if and only if for all objects $C\in\mathbb{C}$, the function 
$f(C)\colon X(C)\rightarrow Y(C)$ of sets has $\kappa$-small fibers.

Given a small simplicial category $\mathbf{C}$, we also denote the cardinality of $\mathbf{C}$ by
\[|\mathbf{C}|:=\sum_{C,C\sprime}|\mathrm{Hom}_{\mathbf{C}}(C,C\sprime)|\]
where the cardinality of the hom-objects $\mathbf{C}(C,C\sprime)\in\mathbf{S}$ is given by the cardinality 
of presheaves defined above.
Accordingly, given a regular cardinal $\kappa>|\mathbf{C}|$, a simplicial presheaf $X\in\mathrm{sPsh}(\mathbf{C})$ is
$\kappa$-small if all its values $X(C)$ are $\kappa$-small. A simplicial natural transformation
$f\colon X\rightarrow Y$ in $\mathrm{sPsh}(\mathbf{C})$ is $\kappa$-small if all its pullbacks along maps
$Z\rightarrow Y$ with $\kappa$-small domain $Z$ are $\kappa$-small simplicial presheaves.

\begin{remark}\label{remsmallptwise}
Again, a map $f\colon X\rightarrow Y$ in $\mathrm{sPsh}(\mathbf{C})$ is $\kappa$-small if and only if for all $C\in\mathbf{C}$, the map $f(C)\colon X(C)\rightarrow Y(C)$ is a $\kappa$-small map of simplicial sets. 
\end{remark}

The category $\mathrm{sPsh}(\mathbf{C})$ is locally finitely presentable (\cite[Examples 3.4, Proposition 7.5]{kellyfinlimits}), 
generated by (finite colimits of) the objects $yC\otimes\Delta^n$ for $C\in\mathbf{C}$ and $n\geq 0$ which we refer to in the 
following as the generators. When an ordinary category $\mathbb{C}$ is considered as a discrete simplicial category, we have an 
obvious isomorphism between $\mathrm{sPsh}(\mathbb{C})$ and the set-valued presheaf category $\widehat{\mathbb{C}\times\Delta^{op}}$.

\begin{notation}
To develop a general theory of accessible categories, Ad\'{a}mek and Rosick\'{y} introduce for pairs of regular cardinals $\mu<\kappa$ 
the \emph{sharply larger} relation (\cite[Definition 2.12]{adamekrosicky}) and a special case ``$\ll$'' in
\cite[Example 2.13.(4)]{adamekrosicky} which is used as well in \cite[Definition 5.4.2.8]{luriehtt} to develop a theory of accessible 
$(\infty,1)$-categories. Here, $\mu\ll\kappa$ if for all cardinals $\kappa_0<\kappa$ and $\mu_0<\mu$, also $\kappa_0^{\mu_0}<\kappa$.
\end{notation}

The order ``$\ll$'' is chosen in such a way that whenever $\mu\ll\kappa$ holds, then $\mu<\kappa$ and
$\mu$-accessibility of a quasi-category $\mathcal{C}$ implies $\kappa$-accessibility of $\mathcal{C}$
(\cite[Proposition 5.4.2.11]{luriehtt}). As noted in \cite{luriehtt}, the order is unbounded in the class of regular cardinals 
as for any regular cardinal $\mu$ we have $\mu\ll\mathrm{sup}\{\tau^{\mu}\mid\tau<\mu\}^+$.
In particular, we always find a regular cardinal sharply larger than a given regular $\mu$.
In fact, if $\mu$ is regular, then $\mu^+$ is already sharply larger than $\mu$ (\cite[Examples 2.13.(2)]{adamekrosicky}). Whenever
$\lambda<\mu$ is a regular cardinal and $\mu\ll\kappa$, then also $\lambda\ll\kappa$. Thus, for any set $X$ of cardinals there is a 
regular cardinal $\mu$ such that $\kappa\ll\mu$ for all $\kappa\in X$.

Recall that an object $C$ in an accessible category $\mathbb{C}$ is \emph{$\kappa$-compact} if its associated 
corepresentable preserves $\kappa$-directed colimits. A map $f$ in $\mathbb{C}$ is \emph{relatively $\kappa$-compact} if the pullback of 
$f$ along any map with $\kappa$-compact domain in $\mathcal{C}$ is itself again $\kappa$-compact.

\begin{lemma}\label{lemmacompactissmall}\mbox{ }
Let $\mathbf{C}$ be a small simplicial category and $\kappa\gg|\mathbf{C}|$ an infinite regular cardinal. Then
\begin{enumerate}
\item An object $X\in \mathrm{sPsh}(\mathbf{C})$ is $\kappa$-compact if and only if it is $\kappa$-small.
\item A map $f\in\mathrm{sPsh}(\mathbf{C})$ is relatively $\kappa$-compact if and only if it is $\kappa$-small.
\end{enumerate}
\end{lemma}
\begin{proof}
Let $\mathbf{C}$ be a small simplicial category.
For Part 1, recall that a presheaf $X$ is $\kappa$-compact if and only 
if it is a retract of a $\kappa$-small directed colimit of finite colimits of the generators $yC\otimes\Delta^n$ via
\cite[Remark 2.15]{adamekrosicky}. But $\kappa\gg|\mathbf{C}|$ being infinite and regular implies that all generators 
$yC\otimes\Delta^n$ are $\kappa$-small, and hence so are all their finite colimits. Hence, every $\kappa$-compact presheaf $X$ is a 
subobject of a $\kappa$-small colimit of
$\kappa$-small presheaves and hence $\kappa$-small. Vice versa, every simplicial presheaf $X$ is the colimit of its canonical diagram 
of generators $yC\otimes\Delta^n$; 
whenever $X$ is $\kappa$-small, so is the associated canonical 
diagram by the Yoneda Lemma. Thus, 
$X$ is a $\kappa$-small colimit of the generators $yC\otimes\Delta^n$. Closing the generators under finite colimits gives a 
description of $X$ as a $\kappa$-small directed colimit of $\kappa$-compact objects, so $X$ is $\kappa$-compact again 
by \cite[Remark 2.15]{adamekrosicky}. Part 2 follows directly from Part 1 by definition.
\end{proof}

\begin{lemma}\label{lemmaadjpressmall}
Let $\mathbb{C}$ and $\mathbb{D}$ be locally presentable categories and let 
\[\xymatrix{F\colon\mathbb{C}\ar@<.5ex>[r] & \mathbb{D}\colon G\ar@<.5ex>[l]}\]
be an adjoint pair. Let $\kappa$ be a regular cardinal such that 
\begin{enumerate}
\item both $F$ and $G$ preserve $\kappa$-compact objects,
\item $\kappa$-compact objects in $\mathbb{C}$ are closed under fiber products.
\end{enumerate}
Then $G$ preserves relatively $\kappa$-compact maps.
\end{lemma}

\begin{proof}
Straight-forward.
\end{proof}

\begin{remark}\label{remsmalllex}
Although the class of $\mu$-compact objects in a locally $\kappa$-presentable category is not necessarily closed under fiber products 
for all $\mu\gg\kappa$, the class of such $\mu$ is unbounded in the class of regular cardinals sharply larger than $\kappa$ (\cite[Proposition 4.5]{shulmanuniverses}). It nevertheless is a trivial property in the case of (simplicial) presheaf 
categories whenever $\mu$ is infinite by Lemma~\ref{lemmacompactissmall}, since $\mu$-smallness of sets is closed under finite limits.
\end{remark}

\begin{corollary}\label{coradjpressmall}
Let $\mathbf{C}$ and $\mathbf{D}$ be small simplicial categories and let 
\[\xymatrix{F\colon\mathrm{sPsh}(\mathbf{C})\ar@<.5ex>[r] & \mathrm{sPsh}(\mathbf{D})\colon G\ar@<.5ex>[l]}\]
be an adjoint pair. Let $\kappa\gg\mathrm{max}(|\mathbf{C}|,|\mathbf{D}|)$ be regular and suppose both $F$ and $G$ preserve $\kappa$-small 
objects. Then $G$ preserves $\kappa$-small maps.
\end{corollary}

\begin{proof}
Follows directly from Lemma~\ref{lemmacompactissmall}.2, Lemma~\ref{lemmaadjpressmall} and Remark~\ref{remsmalllex}.
\end{proof}

The aim of this section is to compare this ordinary notion of compactness in a combinatorial model category $\mathbb{M}$
with the notion of compactness in its underlying $(\infty,1)$-category $\mathrm{Ho}_{\infty}(\mathbb{M},\mathcal{W})$
(\cite[Definition 5.3.4.5, Definition 6.1.6.4]{luriehtt}). The validity of this comparison was addressed in a question posted in 
\cite{mocompactobjects} by Shulman; for objects it is given in Proposition~\ref{lemmacomparesmallobjects} in a special case, and in 
Corollary~\ref{corcmptduggerpres} in general. For maps it is given in Theorem~\ref{thmcmptduggerpres} in case $\mathbb{M}$ is (the 
left Bousfield localization of) a simplicial presheaf category equipped with the projective model structure, and in 
Theorem~\ref{thmcmptinj} in case it is such equipped with the injective model structure. An argument for the objectwise statement was 
outlined by Lurie in the same post, which in one direction coincides with our proof given in Proposition~\ref{lemmacomparesmallobjects}. Before we state the theorems, we make the following ad hoc construction, give one 
auxiliary folkore lemma, and define a simple axiomatic framework of minimal fibrations in a general model category that will provide a 
convenient set-up to state intermediate results in.\\

Given a $\lambda$-accessible quasi-category $\mathcal{C}$ with generating set $A$ and a regular cardinal
$\mu\geq\lambda$, define the full subcategory $J^{\mu}\subseteq\mathcal{C}$ recursively 
as follows. Let
\[J^{\mu,0}_0:=A\]
and $J^{\mu,0}$ be the full subcategory of $\mathcal{C}$ generated by $J^{\mu,0}_0$. 
Whenever $\beta<\mu$ is a limit ordinal, let
\[J^{\mu,\beta}_0=\bigcup_{\alpha<\beta}J^{\mu,\alpha}_0\]
and $J^{\mu,\beta}$ the full subcategory generated by $J^{\mu,\beta}_0$. On successors, given $J^{\mu,\alpha}$, let
\begin{equation}\label{deffreecompacts}
J^{\mu,\alpha+1}_0:=\{\mathrm{colim}F\mid F\colon I\rightarrow J^{\mu,\alpha},I\in\mathrm{QCat}\text{ is $\mu$-small and $\lambda$-filtered}\}
\end{equation}
(so we choose a set of representatives $V_{\mu}^{\Delta^{op}}$ for $\mu$-small simplicial sets) and
$J^{\mu,\alpha+1}$ be the corresponding full subcategory. Eventually, we define the full subcategory
$J^{\mu}$ of $\mathcal{C}$ to have the set of objects
\[J^{\mu}_0:=\bigcup_{\alpha<\mu}J^{\mu,\alpha}_0.\] 

The following lemma is noted in \cite[Section 5.4.2]{luriehtt}, and is a generalization of the corresponding 1-categorical 
statement for accessible categories (\cite[Remark 2.15]{adamekrosicky}).

\begin{lemma}\label{lemmapreservecompact}
Let $\mathcal{C}$ and $\mathcal{D}$ be presentable quasi-categories.
\begin{enumerate}
\item Suppose $\mathcal{C}$ is $\lambda$-presentable. Then, for every regular $\mu\gg\lambda$,
the $\mu$-compact objects in $\mathcal{C}$ are, up to equivalence, exactly the retracts of objects in
$J^{\mu}$.
\item Let $F\colon\mathcal{C}\rightarrow\mathcal{D}$ be an accessible functor. Then there is a cardinal $\mu$ 
such that $F$ preserves $\kappa$-compact objects for all regular $\kappa\gg\mu$.
\end{enumerate}
\end{lemma}
\begin{proof}
See \cite[Lemma 8.3.4]{thesis}.
\end{proof}

\begin{notation}\label{notationlarge}
The following group of statements will in each case claim that a certain comparison holds for all $\kappa$ 
``sufficiently large'' or ``large enough''. That means in each case there is a cardinal $\mu$ such that for all
$\kappa\gg\mu$ the given statement holds true. As we are not interested in a precise formula for the lower bound
$\mu$, we generally will not make the cardinal $\mu$ explicit. Instead, we note that we will have to impose the 
condition on $\kappa$ to be ``large enough'' only finitely often and eventually take the corresponding supremum.
\end{notation}

\begin{definition}\label{defminimalfibs}
Let $\mathbb{M}$ be a model category. Say $\mathbb{M}$ has a \emph{theory of minimal fibrations} if there is a pullback stable class
$\mathcal{F}^{\mathrm{min}}_{\mathbb{M}}$ of fibrations in $\mathbb{M}$ -- the class of \emph{minimal fibrations} -- such that the following hold.
\begin{enumerate}
\item Let $p\colon X\twoheadrightarrow Y$ and $q\colon X\sprime\twoheadrightarrow Y$ be minimal fibrations. Then every weak equivalence 
between $X$ and $X\sprime$ over $Y$ is an isomorphism.
\item For every fibration $p\colon X\twoheadrightarrow Y$ in $\mathbb{M}$ there is an acyclic cofibration
$M\overset{\sim}{\hookrightarrow}X$ such that the restriction $M\rightarrow Y$ is a minimal fibration.
\end{enumerate}
\end{definition}

\begin{lemma}\label{lemmacmpt0}
Let $\mathbb{M}$ be a model category with a theory of minimal fibrations. Let $T$ be a class of maps in $\mathbb{M}$ such that the left Bousfield localization $\mathcal{L}_T \mathbb{M}$ exists. Then the model category $\mathcal{L}_{T}\mathbb{M}$ has a theory of 
minimal fibrations.
\end{lemma}

\begin{proof}
Given a model category $\mathbb{M}$ and a class $T$ of maps in $\mathbb{M}$ as stated, simply define the class
$\mathcal{F}^{\mathrm{min}}_{T}$ of minimal fibrations in $\mathcal{L}_T\mathbb{M}$ to be the class of 
fibrations in $\mathcal{L}_T\mathbb{M}$ which are minimal fibrations in $\mathbb{M}$. Pullback stability of
$\mathcal{F}^{\mathrm{min}}_{T}$ is immediate. Property 1 follows readily, as $T$-local weak equivalences between
$T$-local fibrations are weak equivalences in $\mathbb{M}$ itself. For Property 2, let
$p\colon X\twoheadrightarrow Y$ be a fibration in $\mathcal{L}_T\mathbb{M}$. By the assumption that $\mathbb{M}$ has a 
theory of minimal fibrations, there is an acyclic cofibration $M\overset{\sim}{\hookrightarrow} X$ in $\mathbb{M}$ such 
that the restriction $M\twoheadrightarrow Y$ is a minimal fibration in $\mathbb{M}$. But $M\rightarrow X$ is a weak 
equivalence from the fibration $M\twoheadrightarrow Y$ to the fibration $p\colon X\twoheadrightarrow Y$ over $Y$. The 
latter is a fibration in $\mathcal{L}_T\mathbb{M}$ and it hence follows by \cite[Proposition 3.4.6]{hirschhorn03} that
$M\twoheadrightarrow Y$ is a fibration in $\mathcal{L}_T\mathbb{M}$, too. 
\end{proof}

\begin{proposition}\label{lemmacomparesmallobjects}
Let $\mathbb{M}$ be a combinatorial model category. 
\begin{enumerate}
\item For all sufficiently large regular cardinals $\kappa$, an object $C$ in $\mathrm{Ho}_{\infty}(\mathbb{M})$ 
is $\kappa$-compact if there is a $\kappa$-compact object $D\in\mathbb{M}$ such that $C\simeq D$ in
$\mathrm{Ho}_{\infty}(\mathbb{M})$. 
\item Suppose $\mathbb{M}$ has a theory of minimal fibrations. Then the converse of Part 1 holds. 
\end{enumerate}
\end{proposition}

\begin{proof}
For Part 1 and $\kappa$ large enough, $\kappa$-filtered colimits in $\mathbb{M}$ are homotopy colimits and the $\kappa$-compact 
objects in $\mathbb{M}$ are exactly the $\kappa$-compact objects in the quasi-category $N(\mathbb{M})$. So the localization
$N(\mathbb{M})\rightarrow\mathrm{Ho}_{\infty}(\mathbb{M})$ preserves $\kappa$-filtered colimits and hence 
is $\kappa$-accessible. The statement now follows from Lemma \ref{lemmapreservecompact}.

For Part 2, we note that by our assumption and by Dugger's presentation theorem for combinatorial model categories
\cite[Theorem 1.1]{duggersmallpres} it suffices to show the statement for objects $C\in J^{\kappa}$ on the one hand and left Bousfield 
localizations of simplicial presheaf categories $\mathrm{sPsh}(\mathbb{C})$ on the other. Indeed, given a combinatorial model category 
$\mathbb{M}$ together with a category $\mathbb{C}$, a set $T\subset \mathrm{sPsh}(\mathbb{C})$ of arrows and a Quillen-equivalence of 
the form
\[\xymatrix{\mathcal{L}_T(\mathrm{sPsh}(\mathbb{C}))_{\mathrm{proj}}\ar@<.5ex>[r]^(.7)L & \mathbb{M},\ar@<.5ex>[l]^(.3)R}\]
suppose we have shown the statement for all $C\in J^{\kappa}$ and all $\kappa$ large enough in the case of
$\mathcal{L}_T(\mathrm{sPsh}(\mathbb{C}))_{\mathrm{proj}}$. Then, as both categories $\mathbb{M}$ and $\mathrm{sPsh}(\mathbb{C})$ are 
presentable, we find $\kappa\gg|\mathbb{C}|$ large enough such that the right adjoint $R$ preserves $\kappa$-compact objects. 
Certainly the derived functors $\mathbb{L}L$ and $\mathbb{R}R$ preserve $\kappa$-compactness in $\mathrm{Ho}_{\infty}(\mathbb{M})$, so 
whenever an object $C\in\mathrm{Ho}_{\infty}(\mathbb{M})$ is contained in $J^{\kappa}$, we may choose a $\kappa$-compact presheaf
$D\in \mathrm{sPsh}(\mathbb{C})$ weakly equivalent to $\mathbb{R}RX$. Without loss of generality $D$ is 
cofibrant by \cite[Proposition 2.3.(iii)]{duggersmallpres} and so $L(D)$ is $\kappa$-compact in $\mathbb{M}$ and  
presents $C$ in $\mathrm{Ho}_{\infty}(\mathbb{M})$.

Now, every $\kappa$-compact object $A\in\mathrm{Ho}_{\infty}(\mathbb{M})$ is the 
retract of an object $C\in J^{\kappa}$ by Lemma~\ref{lemmapreservecompact}.1. We thus may present $A$ by a bifibrant object 
$B\in\mathbb{M}$, and $C$ by a $\kappa$-compact bifibrant object $D\in\mathbb{M}$ again via
\cite[Proposition 2.3.(iii)]{duggersmallpres} and by the above. This yields a map $j\colon B\rightarrow D$ with homotopy retract 
$r\colon D\rightarrow B$ in $\mathbb{M}$. Pick a minimal fibrant object $\iota\colon M\hookrightarrow B$. Since every acyclic 
cofibration between fibrant objects allows a retract $\rho$ itself, we see that $M$ is a homotopy retract of $D$. Hence, the 
composition $(\rho r)(j\iota)$ is homotopic to the identity $1_M$ and thus a homotopy equivalence (\cite[Theorem 1.2.10.(iv)]{hovey}). 
It follows that it is an isomorphism in virtue of minimality of $M$. Thus, $M$ is a retract of $D$ and as such
$\kappa$-compact in $\mathbb{M}$ itself.

Therefore, assume $\mathbb{M}=\mathcal{L}_T(\mathrm{sPsh}(\mathbb{C}))_{\mathrm{proj}}$, and suppose
$C\in\mathrm{Ho}_{\infty}(\mathbb{M})$ is contained in $J^{\kappa}$.
The representatives for the colimits in the
construction of $(J^{\kappa,\alpha}|\alpha<\kappa)$ can be chosen to be homotopy colimits of strict diagrams
$F\colon I\rightarrow\mathbb{M}$ for $\kappa$-small categories $I$ by \cite[Proposition 4.2.3.14]{luriehtt} and
\cite[Proposition 1.3.4.25]{lurieha}. Hence, they can be computed according to the Bousfield-Kan formula
\[
\mathrm{hocolim}F=\mathrm{coeq}\left(
\xymatrix{
\coprod\limits_{i\rightarrow j}F(i)\otimes N(j/I)^{\mathrm{op}}
\ar@<.5ex>[r]\ar@<-.5ex>[r] & \coprod\limits_i F(i)\otimes N(i/I)^{\mathrm{op}}
}
\right)\]
because $\mathbb{M}=\mathcal{L}_{T}\mathrm{sPsh}(\mathbb{C})_{\mathrm{proj}}$ is a simplicial model 
category (\cite[Example 18.3.6]{hirschhorn03}). But this representative of the homotopy colimit is $\kappa$-compact whenever $I$ is
$\kappa$-small and furthermore each $F(i)$ for $i\in I$ in $\mathbb{M}$ is $\kappa$-compact. Hence, by induction, every object
$C\in J^{\kappa}$ is presented by a $\kappa$-compact object $D$ in $\mathrm{sPsh}(\mathbb{C})$.
\end{proof}

In the following we generalize Proposition~\ref{lemmacomparesmallobjects} to relatively $\kappa$-compact maps. 

\begin{proposition}\label{propcmpct1}
Let $\mathbb{M}$ be a combinatorial model category. 
\begin{enumerate}
\item Suppose the converse of Proposition~\ref{lemmacomparesmallobjects}.1 holds in $\mathbb{M}$. Then for all sufficiently large regular 
cardinals $\kappa$, a morphism $f\colon C\rightarrow D$ in $\mathrm{Ho}_{\infty}(\mathbb{M})$ is relatively $\kappa$-compact if there is a 
relatively $\kappa$-compact fibration $p\in\mathbb{M}$ between fibrant objects such that $p\simeq f$ in $\mathrm{Ho}_{\infty}(\mathbb{M})$. 
\item Suppose $\mathbb{M}$ has a theory of minimal fibrations and $\kappa$-compact objects in $\mathbb{M}$ are closed under fiber products. 
Then the converse of Part 1.\ holds. 
\end{enumerate}
\end{proposition}

\begin{proof}
For Part 1, let $p\colon X\twoheadrightarrow Y$ be a relatively $\kappa$-compact fibration between fibrant objects in
$\mathbb{M}$. Let $g\colon A\rightarrow Y$ be a map in $\mathrm{Ho}_{\infty}(\mathbb{M})$ with $\kappa$-compact domain. In order to show that 
the pullback of $X$ along $g$ 
is $\kappa$-compact in $\mathrm{Ho}_{\infty}(\mathbb{M})$, by assumption we can present $A$ by a
$\kappa$-compact object $A\sprime$ in $\mathbb{M}$. Without loss of generality $A\sprime$ is bifibrant by 
\cite[Proposition 2.3.(iii)]{duggersmallpres}, so we obtain a map $g\sprime\colon A\sprime\rightarrow Y$ presenting $g$.
Also the pullback $(g\sprime)^{\ast}X$ is a homotopy pullback and it is $\kappa$-compact in $\mathbb{M}$ by assumption. Hence, it is
$\kappa$-compact in $\mathrm{Ho}_{\infty}(\mathbb{M})$ again by Proposition~\ref{lemmacomparesmallobjects}. This shows that $p$ is 
relatively $\kappa$-compact in $\mathrm{Ho}_{\infty}(\mathbb{M})$.

For Part 2, assume that $f\colon C\rightarrow D$ is relatively
$\kappa$-compact in $\mathrm{Ho}_{\infty}(\mathbb{M})$ and $p\colon X\twoheadrightarrow Y$ is a fibration in
$\mathbb{M}$ such that $Y$ is 	fibrant in $\mathbb{M}$ and $p\simeq f$ in $\mathrm{Ho}_{\infty}(\mathbb{M})$. By
assumption there is an acyclic cofibration $M\overset{\sim}{\hookrightarrow} X$ such that the restriction
$m\colon M\twoheadrightarrow Y$ of $p$ is a minimal fibration. As $m$ and $p$ 
are homotopy equivalent over $Y$, the fibration $m$ is relatively $\kappa$-compact in
$\mathrm{Ho}_{\infty}(\mathbb{M})$, too. We want to show that $m$ is a relatively $\kappa$-compact fibration in $\mathbb{M}$.  
Therefore, let $g\colon Z\rightarrow Y$ be a map for some $\kappa$-compact object $Z\in\mathbb{M}$; we have to show that the pullback 
\[\xymatrix{
g^{\ast}M\ar[r]\ar[d]_{g^{\ast}m}\ar@{}[dr]|(.3){\pbs} & M\ar[d] \\
Z\ar[r]_g & Y
}\]
is a $\kappa$-compact object in $\mathbb{M}$ as well.
By \cite[Proposition 2.3.(iii)]{duggersmallpres} there is a $\kappa$-compact fibrant replacement $RZ$ of $Z$. Since the object $Y$ itself is fibrant, we obtain an extension $g\sprime\colon RZ\rightarrow Y$ of $g$ along the acylic cofibration
$Z\overset{\sim}{\hookrightarrow}RZ$ and hence a factorization of the following form.
\[\xymatrix{
g^{\ast}M\ar[rr]\ar@{->>}[dd]_{g^{\ast}m}\ar[dr] & & M\ar@{->>}[dd]^m\\
 & (g\sprime)^{\ast} M\ar[ur]\ar@{->>}[dd] & \\
 Z\ar[rr]|(.52)\hole_(.4)g\ar@{^(->}[dr]_{\sim} & & Y \\
  & RZ\ar[ur]_{g\sprime} & 
}\]
All three faces of the diagram are pullback squares, and by assumption $\kappa$-compact objects in $\mathbb{M}$ are closed under 
fiber products. Hence, in order to show that the object $g^{\ast}M\in\mathbb{M}$ is $\kappa$-compact, it suffices to show that the 
object $(g\sprime)^{\ast}M\in\mathbb{M}$ is $\kappa$-compact.

As $RZ$ is $\kappa$-compact in $\mathbb{M}$, it also is $\kappa$-compact in the underlying quasi-category
$\mathrm{Ho}_{\infty}(\mathbb{M})$ by Proposition~\ref{lemmacomparesmallobjects}, and hence so is the (homotopy-)pullback
$(g\sprime)^{\ast}M$ by our assumption on the morphism $f$ that we started with. 

By Proposition~\ref{lemmacomparesmallobjects} and \cite[Proposition 2.3.(iii)]{duggersmallpres} we find a cofibrant $\kappa$-compact 
object $X$ together with a weak equivalence $e\colon X\rightarrow (g\sprime)^{\ast}M$. The composition
$(g\sprime)^{\ast}m\circ e\colon X\rightarrow RZ$ is a map between $\kappa$-compact objects in $\mathbb{M}$, and so we find a 
factorization $X\overset{\sim}{\hookrightarrow} RX\twoheadrightarrow RZ$ such that $RX$ is $\kappa$-compact as well again by 
\cite[Proposition 2.3.(iii)]{duggersmallpres}. We obtain a weak equivalence $RX\rightarrow (g\sprime)^{\ast}M$ between the 
respective fibrations over $RZ$ as a lift to the resulting square. 

Since $\mathbb{M}$ has a theory of minimal fibrations, there is an acyclic cofibration
$j\colon N\overset{\sim}{\hookrightarrow}RX$ such that the restriction $n\colon N\twoheadrightarrow RZ$ of
$RX\twoheadrightarrow RZ$ is a minimal fibration. Since $j$ is an acyclic cofibration between fibrations, it has a 
retraction, and so $N$ is still $\kappa$-compact in $\mathbb{M}$. But the composition of weak equivalences 
$N\simeq (g\sprime)^{\ast}M$ over $RZ$ is a weak equivalence between minimal fibrations and hence is an isomorphism. Thus,
$(g\sprime)^{\ast}M$ is $\kappa$-compact in $\mathbb{M}$.
\end{proof}

\begin{corollary}\label{corcmpctezreedy}
Let $\mathbb{P}$ be an Eilenberg-Zilber category in the sense of \cite[Section 2.1]{cisinski} and
$\mathbb{M}=\mathrm{sPsh}(\mathbb{P})_{\mathrm{inj}}$ be the category of simplicial presheaves on $\mathbb{P}$ equipped 
with the injective model structure. Then for all sufficiently large regular cardinals $\kappa$, a morphism
$f\in\mathrm{Ho}_{\infty}(\mathbb{M})$ is relatively $\kappa$-compact if and only if there is a $\kappa$-small fibration 
$p\in\mathbb{M}$ between fibrant objects such that $p\simeq f$ in $\mathrm{Ho}_{\infty}(\mathbb{M})$.
\end{corollary}

\begin{proof}
The model category $\mathbb{M}$ supports a theory of minimal fibrations as shown in \cite[2.13-2.16]{cisinski}, and $\kappa$-small 
objects in (simplicial) presheaf categories are closed under fiber products for infinite $\kappa$. Thus the statement follows from 
Proposition~\ref{propcmpct1} for $\mathbb{M}=\mathrm{sPsh}(\mathbb{P})_{\mathrm{inj}}$.
\end{proof}

\subsection{The projective case}

We now make use of the observations in Section~\ref{sectransferskeleton} to generalize Corollary~\ref{corcmpctezreedy} 
to the category of simplicial presheaves over arbitrary small simplicial categories.

Therefore, we want to make use of \cite[Proposition 5.10, Corollary 6.5]{duggerunivhtytheories} which shows that any zig-zag
$\mathbb{M}_1\leftarrow\dots\rightarrow\mathbb{M}_n$ of Quillen-equivalences between combinatorial model categories can be reduced to 
a single Quillen-equivalence whenever either $\mathbb{M}_1$ or $\mathbb{M}_n$ is a ``standard presentation'' of the form
$\mathcal{L}_T\mathrm{sPsh}(\mathbb{C})_{\mathrm{proj}}$ for some small category $\mathbb{C}$ and some set of maps 
$T\subset\mathrm{sPsh}(\mathbb{C})$. In our case however, we we wish to start with model categories of the form
$\mathcal{L}_T\mathbf{sPsh}(\mathbf{C})_{\mathrm{proj}}$ for general small \emph{simplicial} categories $\mathbf{C}$ instead.
Therefore, we show a simplicially enriched version of
\cite[Proposition 5.10]{duggerunivhtytheories} first.

\begin{proposition}\label{lemmaduggergen}
Let $\mathbf{C}$ be a small simplicial category, and $\mathbb{M}$, $\mathbb{N}$ be simplicial model categories together 
with a simplicial Quillen-equivalence $(L,R)\colon\mathbb{N}\xrightarrow{\sim}\mathbb{M}$. Let $T$ be a class of arrows 
in $\mathrm{sPsh}(\mathbf{C})$ such that its left Bousfield localization exists, and let
$(F,G)\colon\mathcal{L}_T\mathbf{sPsh(\mathbf{C})}_{\mathrm{proj}}\rightarrow\mathbb{M}$ be a simplicial Quillen pair. 
Then there is a simplicial Quillen pair
$(F\sprime,G\sprime)\colon\mathcal{L}_T\mathbf{sPsh(\mathbf{C})}_{\mathrm{proj}}\rightarrow\mathbb{N}$ such that the 
functors $L\circ F\sprime$ and $F$ are Quillen-homotopic in the sense of \cite[Definition 5.9]{duggerunivhtytheories}. 
In other words, simplicial Quillen pairs with domain $\mathcal{L}_T\mathbf{sPsh}(\mathbf{C})_{\mathrm{proj}}$ can be 
lifted up to homotopy along Quillen-equivalences.
\end{proposition}
\begin{proof}
Let $\lambda$ and $\rho$ denote cofibrant and fibrant replacements respectively and $\mathbb{L}=\lambda^{\ast}$,
$\mathbb{R}=\rho^{\ast}$ denote their associated left and right derivations of functors. Let the composition
\[\lambda\mathbb{R}(R)Fy\colon\mathbf{C}\rightarrow\mathbb{N}\]
be denoted by $p$. Note that the left and right derivation $\mathbb{L}$ and $\mathbb{R}$ of simplicial functors may be 
chosen to be simplicial again by \cite[Corollary 13.2.4]{riehlcthty}, thus $p$ is a simplicial functor and we can 
consider the simplicially enriched left Kan extension
\[\xymatrix{
\mathbf{C}\ar[d]_y\ar[r]^{p} & \mathbb{N}. \\
\mathbf{sPsh}(\mathbf{C})\ar@{-->}[ur]_{\mathrm{Lan}_yp} & 
}\]
We claim that $F\sprime:=\mathrm{Lan}_yp$ is the left Quillen 
functor we are looking for. First, let us construct the Quillen homotopies connecting $L\circ F\sprime$ and $F$.

Recall that, as explained for instance in \cite[4.31]{kellybook}, for every presheaf
$X\in\mathbf{sPsh}(\mathbf{C})$ the object $\mathrm{Lan}_yp(X)$ is the colimit of $p$ weighted by $X$, i.e.\
\[F\sprime=\blank\star p.\]
The left Quillen functor $L\colon\mathbb{N}\rightarrow\mathbb{M}$ is a left adjoint and hence preserves weighted
colimits, thus we have that $L\circ F\sprime\cong\blank\star Lp$. Furthermore, by \cite[Theorem 3.3]{gambinosmplwghts} 
the weighted colimit functor
\[\blank\star\blank\colon\mathbf{sPsh}(\mathbf{C})_{\mathrm{proj}}\times[\mathbf{C},\mathbb{M}]_{\mathrm{inj}}\rightarrow\mathbb{M}\]
is a left Quillen bifunctor. In particular, for cofibrant presheaves $X\in\mathbf{sPsh}(\mathbf{C})_{\mathrm{proj}}$ the 
$X$-weighted colimit
\[X\star\blank\colon[\mathbf{C},\mathbb{M}]_{\mathrm{inj}}\rightarrow\mathbb{M}\]
is a left Quillen functor. But both $Fy$ and $Lp\cong \mathbb{L}(L)\mathbb{R}(R)Fy$ are cofibrant objects in
$[\mathbf{C},\mathbb{M}]_{\mathrm{inj}}$: the former because representables are projectively cofibrant 
and $F$ preserves cofibrant objects, and the latter because $L$ preserves cofibrant objects. Thus, if
$\rho_{Fy}\colon Fy\rightarrow r(Fy)$ denotes an injective fibrant replacement of $Fy$, the counit
$\varepsilon_{r(Fy)}\colon Lp\Rightarrow r(Fy)$ of the Quillen-equivalence $(L,R)$ induces a span of natural weak 
equivalences between the cofibrant objects $Lp$, $r(Fy)$ and $Fy$. Thus, for cofibrant presheaves
$X\in(\mathrm{sPsh}(\mathbf{C}))_{\mathrm{proj}}$, we 
obtain a zig-zag of natural weak equivalences between $X\star Lp$ and $X\star Fy$. But
$\blank\star Fy$ is just $F$ (by \cite[4.51]{kellybook}), so we have constructed a span of Quillen homotopies between
$L\circ F\sprime$ and $F$.

Second, the fact that $F\sprime\colon \mathbf{sPsh}(\mathbf{C})_{\mathrm{proj}}\rightarrow\mathbb{N}$ is a left Quillen 
functor with right adjoint $G\sprime(N)=\mathbb{N}(p\blank,N)$ was basically already shown above (following for 
instance, as it were, from \cite[Theorem 3.3.]{gambinosmplwghts}).

We are left to show that, third, the Quillen pair
\[(F\sprime,G\sprime)\colon\mathbf{sPsh}(\mathbf{C})_{\mathrm{proj}}\rightarrow\mathbb{N}\]
descends to the localization at $T$ whenever $F$ does so. That is, we have to show that every arrow $f\in T$ is mapped
to a weak equivalence by $F\sprime$ in $\mathbb{N}$ assuming every such arrow is mapped to a weak equivalence by $F$ in $\mathbb{M}$. 
Without loss of generality all arrows $f\in T$ have cofibrant domain and codomain. Then, given $f\in T$, the arrow $F(f)$ is a weak 
equivalence in $\mathbb{M}$, and so is
$LF\sprime(f)\in\mathbb{M}$ since $F$ and $LF\sprime$ are Quillen-homotopic. Thus, $\mathbb{R}(R)(LF\sprime(f))$ is a 
weak equivalence in $\mathbb{N}$, but this arrow is weakly equivalent to $F\sprime(f)$ since $(L,R)$ is a Quillen 
equivalence. It follows that $F\sprime(f)$ is a weak equivalence in $\mathbb{N}$ itself.
\end{proof}

\begin{theorem}\label{thmcmptduggerpres}
Let $\mathbf{C}$ be a small simplicial category, $T\subset\mathrm{sPsh}(\mathbf{C})$ be a set of maps and
$\mathbb{M}=\mathcal{L}_T(\mathbf{sPsh}(\mathbf{C}))_{\mathrm{proj}}$. 
Then for all sufficiently large regular cardinals $\kappa$, a morphism $f\in\mathrm{Ho}_{\infty}(\mathbb{M})$ is relatively $\kappa$-compact if 
and only if there is a $\kappa$-small fibration $p\in\mathbb{M}$ between fibrant objects such that $p\simeq f$ in
$\mathrm{Ho}_{\infty}(\mathbb{M})$.
\end{theorem}

\begin{proof}
Let $\mathbf{C}$ be a small simplicial category and $T\subset\mathrm{sPsh}(\mathbf{C})$ be a set of maps. By 
Proposition~\ref{proptransferskeletonproj} we obtain a relative poset $(\mathbb{P},V)$ and a zig-zag of Quillen-equivalences of the 
form
\[\xymatrix{
\mathcal{L}_{y[V]}\mathbf{sPsh}(\mathbb{P})_{\mathrm{inj}}\ar@<-.5ex>[r]_{\mathrm{id}} & \mathcal{L}_{y[V]}\mathbf{sPsh}(\mathbb{P})_{\mathrm{proj}}\ar@<-.5ex>[r]_(.45){\varphi^{\ast}}\ar@<-.5ex>[l]_{\mathrm{id}} & \mathcal{L}_{y[F_{\Delta}(V)]}\mathbf{sPsh}(F_{\Delta}\mathbb{P})_{\mathrm{proj}}\ar@<-.5ex>[l]_(.55){\varphi_!}
}\]
\[\xymatrix{
\mathcal{L}_{y[F_{\Delta}(V)]}\mathbf{sPsh}(F_{\Delta}\mathbb{P})_{\mathrm{proj}}\ar@<.5ex>[r]^(.55){j_!} & \mathbf{sPsh}(\mathcal{L}_{\Delta}(\mathbb{P},V))_{\mathrm{proj}}\ar@<.5ex>[l]^(.45){j^{\ast}}\ar@<.5ex>[r]^(.7){(f_n)_{!}} 
& \dots\ar@<.5ex>[l]^(.3){f_n^{\ast}}\ar@<-.5ex>[r]_(.3){f_1^{\ast}} & \mathbf{sPsh}(\mathbf{C})_{\mathrm{proj}}
\ar@<-.5ex>[l]_(.7){(f_1)_!}.
}\]
This yields a zig-zag of Quillen-equivalences
\[\xymatrix{
\mathcal{L}_{(y[V]\cup\bar{T})}\mathbf{sPsh}(\mathbb{P})_{\mathrm{inj}}\ar@<.5ex>[r] & \dots\ar@<.5ex>[r]\ar@<.5ex>[l] & \mathcal{L}_T\mathbf{sPsh}(\mathbf{C})_{\mathrm{proj}}\ar@<.5ex>[l]
}\]
where $\bar{T}\subset\mathrm{sPsh}(\mathbb{P})$ is obtained from $T\subset\mathrm{sPsh}(\mathbf{C})$ by transferring $T$ along 
the finitely many Quillen-equivalences successively. We denote the union $y[V]\cup\bar{T}\subset\mathrm{sPsh}(\mathbb{P})$
short-handedly by $U$.

By Proposition~\ref{lemmaduggergen} this chain of Quillen-equivalences induces a single Quillen-equivalence
\begin{equation}\label{eququequivdugger}
\xymatrix{
\mathcal{L}_T\mathbf{sPsh}(\mathbf{C})_{\mathrm{proj}}\ar@<.5ex>[r]^F & \mathcal{L}_{U}\mathbf{sPsh}(\mathbb{P})_{\mathrm{inj}}.\ar@<.5ex>[l]^G
}
\end{equation} 
The left Bousfield localization $\mathcal{L}_U\mathbf{sPsh}(\mathbb{P})_{\mathrm{inj}}$ has a theory of minimal fibrations by
Lemma~\ref{lemmacmpt0}. Thus, let $\kappa\gg|\mathbf{C}|,|\mathbb{P}|$ be regular and, first, large enough such that
Corollary~\ref{corcmpctezreedy} applies to $\mathbb{P}$, second, large enough such that Proposition~\ref{propcmpct1} applies 
to $\mathcal{L}_U\mathbf{sPsh}(\mathbb{P})_{\mathrm{inj}}$, and third, large enough such that both $F$ and $G$ preserves $\kappa$-compact 
objects (via Lemma~\ref{lemmapreservecompact} or its ordinary categorical analogon as right adjoints between locally presentable categories 
are accessible again).

Now, let $f\in\mathrm{Ho}_{\infty}(\mathbb{M})$ be relatively $\kappa$-compact. Since the pair 
(\ref{eququequivdugger}) is a Quillen-equivalence, the quasi-category $\mathrm{Ho}_{\infty}(\mathbb{M})$ is equivalent 
to the underlying quasi-category of $\mathcal{L}_U\mathbf{sPsh}(\mathbb{P})_{\mathrm{inj}}$. Then, by
Proposition~\ref{propcmpct1}, there is a $\kappa$-small fibration $p\colon X\twoheadrightarrow Y$ between fibrant 
objects in $\mathcal{L}_U\mathbf{sPsh}(\mathbb{P})_{\mathrm{inj}}$ presenting $f$ in
$\mathrm{Ho}_{\infty}(\mathbb{M})$. By assumption both adjoints preserve $\kappa$-compact objects, so by Corollary~\ref{coradjpressmall} the 
right Quillen functor $G$ preserves $\kappa$-small maps. Thus, $Gp\colon GX\twoheadrightarrow GY$ is a $\kappa$-small fibration between fibrant 
objects presenting $f$ in $\mathrm{Ho}_{\infty}(\mathbb{M})$.

In particular, the converse of Proposition~\ref{lemmacomparesmallobjects}.1 holds in $\mathbb{M}$ for every such $\kappa$, and so the other 
direction follows directly from Proposition~\ref{propcmpct1}.1.
\end{proof}

\begin{corollary}\label{corcmptduggerpres}
Let $\mathbb{M}$ be a combinatorial model category. 
\begin{enumerate}
\item For all sufficiently large regular cardinals $\kappa$, an object $C$ in $\mathrm{Ho}_{\infty}(\mathbb{M})$ 
is $\kappa$-compact if and only if there is a $\kappa$-compact object $D\in\mathbb{M}$ such that $C\simeq D$ in
$\mathrm{Ho}_{\infty}(\mathbb{M})$. 
\item Let $\mathcal{L}_T(\mathrm{sPsh}(\mathbb{C}))_{\mathrm{proj}}$ be the presentation of $\mathbb{M}$ from Dugger's representation 
theorem for combinatorial model categories in \cite{duggersmallpres}. Then for all sufficiently large regular cardinals $\kappa$, 
a morphism $f\in\mathrm{Ho}_{\infty}(\mathbb{M})$ is relatively $\kappa$-compact if and only if there is a $\kappa$-small fibration
$p\in\mathcal{L}_T(\mathrm{sPsh}(\mathbb{C}))_{\mathrm{proj}}$ between fibrant objects such that $p\simeq f$ 
in $\mathrm{Ho}_{\infty}(\mathbb{M})$.
\end{enumerate}
\end{corollary}
\begin{proof}
For Part 1, one direction is exactly Proposition~\ref{lemmacomparesmallobjects}.1. For the other direction, let 
\[\xymatrix{\mathcal{L}_T\mathrm{sPsh}(\mathbb{C})_{\mathrm{proj}}\ar@<.5ex>[r]^(.7)L & \mathbb{M}\ar@<.5ex>[l]^(.3)R}\]
be the Quillen-equivalence from Dugger's representation theorem. Then for $\kappa$ large enough and every $\kappa$-compact object 
$C\in\mathrm{Ho}_{\infty}(\mathbb{M})$, we obtain a $\kappa$-compact object $D\in\mathcal{L}_T\mathrm{sPsh}(\mathbb{C})_{\mathrm{proj}}$ from 
Theorem~\ref{thmcmptduggerpres} which presents $C$. As the left adjoint $L$ preserves $\kappa$-compact objects, we find a $\kappa$-compact 
fibrant replacement of $L(D)$ in $\mathbb{M}$ which presents $C$.

Part 2 is just a special case of Theorem~\ref{thmcmptduggerpres}.
\end{proof}

\begin{remark}\label{remgencombprob}
The reason why in Corollary~\ref{corcmptduggerpres}.2 we don't obtain the comparison result for $\mathbb{M}$ itself is 
that there is no obvious reason why the Quillen-equivalence
\[\xymatrix{\mathcal{L}_T\mathrm{sPsh}(\mathbb{C})_{\mathrm{proj}}\ar@<.5ex>[r]^(.7)L & \mathbb{M}\ar@<.5ex>[l]^(.3)R}\]
given by Dugger's presentation theorem should preserve relatively $\kappa$-compact maps. While the right adjoint certainly 
does preserve such maps, the left adjoint does not seem to exhibit any properties with that respect.
\end{remark}

\subsection{The injective case}

In this section we prove an analogous result for the injective model structure and get rid of the condition on fibrancy 
of the bases whenever the localization is left exact. We will make use of Shulman's results \cite{shulmanuniverses} 
in two ways. Therefore, applied to the special case relevant for this paper, recall the forgetful functor
\begin{align}\label{defforgetfun}
U\colon\mathrm{sPsh}(\mathbf{C})\rightarrow\mathbf{S}^{\mathrm{Ob}(\mathbf{C})}
\end{align}
with right adjoint
\[G\colon \mathbf{S}^{\mathrm{Ob}(\mathbf{C})}\rightarrow\mathrm{sPsh}(\mathbf{C}).\]
The functor $G$ takes objects $W\in\mathbf{S}^{\mathrm{Ob}(\mathbf{C})}$ to the presheaf evaluating an object $C\in\mathbf{C}$ at
\[G(W)(C):=\prod_{C\sprime\in\mathbf{C}}W(C\sprime)^{\mathbf{C}(C\sprime,C)}\in\mathbf{S}.\]
The adjoint pair $(U,G)$ gives rise to a comonad on $\mathrm{sPsh}(\mathbf{C})$ with standard resolution
\[\mathsf{C}_{\bullet}(G,UG,U\blank)\colon\mathrm{sPsh}(\mathbf{C})\rightarrow \mathrm{sPsh}(\mathbf{C})^{\Delta}.\]
The associated cobar construction
$\mathsf{C}(G,UG,U\blank)\colon\mathrm{sPsh}(\mathbf{C})\rightarrow\mathrm{sPsh}(\mathbf{C})$ is then defined as the 
pointwise totalization
\[\mathrm{Tot}(\mathsf{C}_{\bullet}(G,UG,U\blank))=\int_{[n]\in\Delta}(\mathsf{C}_n(G,UG,U\blank))^{\Delta^n}.\] 
A crucial observation of Shulman is that the cobar construction takes (acyclic) projective fibrations to 
pointwise weakly equivalent (acyclic) injective fibrations. More precisely, the natural coaugmentation
$\eta\colon\mathrm{id}\Rightarrow\mathsf{C}(G,UG,U\blank)$ is a pointwise weak equivalence, and the arrow
$\mathsf{C}(G,UG,U p)$ is an (acyclic) injective fibration whenever $p$ is an (acyclic) projective fibration. All this is 
covered in \cite[Section 8]{shulmanuniverses} in much greater generality.
It is not hard to see that the cobar construction preserves $\kappa$-smallness (for $\kappa$ large enough).

\begin{lemma}\label{lemmacobarsmall}
Let $\mathbf{C}$ be a small simplicial category and $f\colon X\rightarrow Y$ be a $\kappa$-small map in
$\mathrm{sPsh}(\mathbf{C})$ for $\kappa$ large enough. Then $\mathsf{C}(G,UG,U f)$ is $\kappa$-small, too.
\end{lemma}

\begin{proof}
The forgetful functor (\ref{defforgetfun}) preserves $\kappa$-smallness of both objects and maps by Remark~\ref{remsmallptwise}. 
Hence, by Lemma~\ref{lemmaadjpressmall}, the right adjoint $G$ preserves $\kappa$-smallness of maps, too. It follows that for every
$\kappa$-small map $f\colon X\rightarrow Y$ in $\mathrm{sPsh(\mathbf{C}})$, the map $\mathsf{C}_{\bullet}(G,UG,Uf)$ of cosimplicial 
objects is levelwise $\kappa$-small. Thus we are only left to show that totalization preserves $\kappa$-smallness of cosimplicial 
objects. But, being a subobject of a countable product of $\kappa$-small simplicial sets, the statement follows.
\end{proof}

Therefore, we directly obtain an analogue of Theorem~\ref{thmcmptduggerpres} for the injective model structure as 
follows.

\begin{proposition}\label{propcmptinj}
Let $\mathbf{C}$ be a small simplicial category, $T\subset\mathrm{sPsh}(\mathbf{C})$ be a set of maps and
$\mathbb{M}=\mathcal{L}_T\mathbf{sPsh}(\mathbf{C})_{\mathrm{inj}}$. 
Then for all sufficiently large regular cardinals $\kappa$, a morphism $f\in\mathrm{Ho}_{\infty}(\mathbb{M})$ is 
relatively $\kappa$-compact if and only if there is a $\kappa$-small fibration
$p\in\mathbb{M}$ between fibrant objects such that $p\simeq f$ in $\mathrm{Ho}_{\infty}(\mathbb{M})$.
\end{proposition}

\begin{proof}
Let $f$ be relatively $\kappa$-compact in $\mathrm{Ho}_{\infty}(\mathbb{M})$. By Theorem~\ref{thmcmptduggerpres} there is 
a $\kappa$-small fibration $p\colon X\twoheadrightarrow Y$ between fibrant objects in
$\mathcal{L}_T\mathrm{sPsh}(\mathbf{C})_{\mathrm{proj}}$ such that $p\simeq f$ in the underlying quasi-category. 
Hence, by Lemma~\ref{lemmacobarsmall} and \cite[Section 8]{shulmanuniverses} the map
\[\mathsf{C}(G,UG,U p)\colon \mathsf{C}(G,UG,U X)\twoheadrightarrow \mathsf{C}(G,UG,U Y)\]
is a $\kappa$-small injective fibration between injectively fibrant objects. But the coaugmentations $\eta_X$ and 
$\eta_Y$ are pointwise weak equivalences in $\mathcal{L}_T\mathbf{sPsh}(\mathbf{C})_{\mathrm{inj}}$, and so the objects
$\mathsf{C}(G,UG,U X)$ and $\mathsf{C}(G,UG,U Y)$ are $T$-local and thus fibrant in $\mathbb{M}$. Hence, the map
$\mathsf{C}(G,UG,U p)$ is a fibration between fibrant objects in $\mathbb{M}$.

The other direction follows immediately from Theorem~\ref{thmcmptduggerpres} since every injective fibration is a 
projective fibration. 
\end{proof}

\begin{remark}\label{remfibext}
Whenever $\mathbb{M}$ satisfies the fibration extension property for relatively $\kappa$-compact maps
(\cite[Definition 2.2.1]{thesis}), we can get rid of the fibrancy condition on the bases of maps in Proposition~\ref{propcmptinj}. 
That is, because in that case every relatively $\kappa$-compact fibration is weakly equivalent to a relatively $\kappa$-compact fibration 
with fibrant base. For example, every left exact left Bousfield localization of $\mathrm{sPsh}(\mathbf{C})_{\mathrm{inj}}$ is a type 
theoretic model topos by \cite[Corollary 8.31, Theorem 10.5]{shulmanuniverses} and hence has univalent universes (with fibrant base) 
for $\kappa$-small fibrations for every regular $\kappa$ large enough (\cite[Theorem 5.22]{shulmanuniverses}). It follows that the class 
$S_{\kappa}$ of $\kappa$-small maps does satisfy the fibration extension property in every left exact left Bousfield localization of
$\mathrm{sPsh}(\mathbf{C})_{\mathrm{inj}}$ (\cite[Lemma 2.2.2]{thesis}).
\end{remark}

\begin{theorem}\label{thmcmptinj}
Let $\mathbf{C}$ be a small simplicial category, and let $T\subset\mathrm{sPsh}(\mathbf{C})$ be a set of maps such that the 
localization $\mathbb{M}=\mathcal{L}_T\mathbf{sPsh}(\mathbf{C})_{\mathrm{inj}}$ is left exact. 
Then for all sufficiently large regular cardinals $\kappa$, a morphism $f\in\mathrm{Ho}_{\infty}(\mathbb{M})$ is 
relatively $\kappa$-compact if and only if there is a $\kappa$-small fibration
$p\in\mathbb{M}$ such that $p\simeq f$ in $\mathrm{Ho}_{\infty}(\mathbb{M})$.
\end{theorem}
\begin{proof}
Immediate by Proposition~\ref{propcmptinj} and Remark~\ref{remfibext}.
\end{proof}

\section{Object classifiers and weak Tarski universes}\label{secuniverses}

We conclude with a comment on the relevance of these results for the $(\infty,1)$-categorical semantics of Homotopy Type Theory.
Let $\mathcal{M}$ be an $\infty$-topos and let $\mathbf{C}$ be a small simplicial category with a set $T$ of 
arrows in $\mathrm{sPsh}(\mathbf{C})$ such that the localization
$\mathbf{sPsh}(\mathbf{C})_{\mathrm{inj}}\rightarrow\mathcal{L}_T\mathbf{sPsh}(\mathbf{C})_{\mathrm{inj}}$ is left exact and presents 
$\mathcal{M}$. Then $\mathbb{M}:=\mathcal{L}_T\mathbf{sPsh}(\mathbf{C})_{\mathrm{inj}}$ is a type theoretic model category as shown in 
\cite[Section 7]{gepnerkock}. Shulman recently has shown in \cite{shulmanuniverses} (among other results) that this presentation
$\mathbb{M}$ in fact can be enhanced to a type theoretic model \emph{topos}, and hence exhibits an infinite sequence of categorical models of 
univalent strict Tarski universes. Furthermore, as for example stated in the Introduction of \cite{gepnerkock}, it is somewhat folklore to 
assume that these categorical models of universes are object classifiers in $\mathcal{M}$, and that more generally the object classifiers 
in $\mathcal{M}$ correspond to categorical models for univalent weak Tarski universes in $\mathbb{M}$.
Here, by a categorical model of a weak Tarski universe we understand a regular (or potentially more specific) cardinal $\kappa$ together with 
a fibration that is weakly universal for the class of $\kappa$-small fibrations. Weak universality of a fibration
$p\colon E\twoheadrightarrow B$ for a class $S$ of fibrations in turn means that $p$ is contained in $S$, that it is univalent, and that for 
all fibrations $q\colon X\twoheadrightarrow Y$ in $S$ there is a map $w\colon X\rightarrow B$ such that $q$ is the homotopy pullback of $p$ 
along $w$. Clearly, every univalent strictly universal fibration is a weakly universal fibration for the same class of maps whenever the 
model category $\mathbb{M}$ is right proper.

Since all fibrant objects in $\mathbb{M}=\mathcal{L}_T\mathbf{sPsh}(\mathbf{C})_{\mathrm{inj}}$ are cofibrant and $\mathbb{M}$ is 
right proper indeed, it is easy to see that a univalent weakly universal fibration for a pullback stable class $S$ of fibrations in
$\mathbb{M}$ yields a classifying object for the class $\mathrm{Ho}_{\infty}[S]$ of 
morphisms in $\mathcal{M}$ and that, vice versa, every classifying object for a pullback stable class $T$ of morphisms 
in $\mathcal{M}$ yields a univalent weakly universal fibration for the class
\[\bar{T}:=\{f\in\mathcal{F}_{\mathbb{M}}\mid f\in\mathrm{Ho}_{\infty}(\mathbb{M})\text{ is in $T$}\}\]
of maps in $\mathbb{M}$. Here, the higher categorical notion of univalence that characterizes object classifiers corresponds to the 
model categorical -- and hence to the syntactical -- notion of univalence by \cite[Proposition 7.12]{gepnerkock}.

There is one such pair of classes $(S,T)$ of maps in each case which is relevant for the construction of strict Tarski 
universes in the internal language of $\mathbb{M}$ on the one hand, and the definition of object classifiers in
$\mathcal{M}$ on the other. That is, given a sufficiently large regular cardinal $\kappa$, the class 
$S_{\kappa}$ of $\kappa$-small fibrations in $\mathrm{sPsh}(\mathbf{C})$ and the class $T_{\kappa}$ of relatively
$\kappa$-compact maps in $\mathcal{M}$.
In the former case, the common constructions of univalent universal fibrations
$\pi_{\kappa}\colon\tilde{U}_{\kappa}\twoheadrightarrow U_{\kappa}$ use various functorial closure properties of 
$S_{\kappa}$ and the fact that an infinite sequence of inaccessible cardinals yields a cumulative hierarchy of universal 
fibrations which are closed under all standard type formers in this way.
In the latter case, \cite[Theorem 6.1.6.8]{luriehtt} characterizes $\infty$-toposes in terms of classifying objects
$p_{\kappa}\colon\tilde{V}_{\kappa}\rightarrow V_{\kappa}$ for $T_{\kappa}$ for all sufficiently large regular cardinals $\kappa$.

While the classifying map $p_{\kappa}\colon\tilde{V}_{\kappa}\rightarrow V_{\kappa}$
lifts to a fibration in $\mathbb{M}$ which is weakly universal for $\bar{T}_{\kappa}$, and $\pi_{\kappa}$ descends to 
a classifying object for the class $\mathrm{Ho}_{\infty}[S_{\kappa}]$, it is a priori unclear whether
$S_{\kappa}\subseteq\bar{T}_{\kappa}$ and $T_{\kappa}\subseteq\mathrm{Ho}_{\infty}[S_{\kappa}]$ hold. In other words, without a 
comparison of relative compactness notions as considered in Section~\ref{secsmallvscompact},
it is not clear whether the categorical construction of (either weak or strict) universal $\kappa$-small fibrations in
$\mathbb{M}$ -- which models Tarski universes in the associated type theory -- also models universes in the underlying 
quasi-category. Theorem~\ref{thmcmptinj} however does show $T_{\kappa}=\mathrm{Ho}_{\infty}[S_{\kappa}]$. In other words, we obtain 
the following corollary.

\begin{corollary}\label{coruniverses}
Let $\mathbb{M}=\mathcal{L}_T\mathbf{sPsh}(\mathbf{C})_{\mathrm{inj}}$ be a model topos, and let $\kappa$ be a sufficiently large 
regular cardinal. Then a relatively $\kappa$-compact map $p\in\mathrm{Ho}_{\infty}(\mathbb{M})$ is a classifying map for all 
relatively $\kappa$-compact maps in $\mathrm{Ho}_{\infty}(\mathbb{M})$ if and only if there is a univalent $\kappa$-small fibration
$\pi\in\mathbb{M}$ which is weakly universal for all $\kappa$-small fibrations in $\mathbb{M}$ such that $p\simeq\pi$ in
$\mathrm{Ho}_{\infty}(\mathbb{M})$.  \qed
\end{corollary}

\begin{remark}\label{remlargecardinals}
Let us finish with a note on the closure under standard type formers of a given Tarski universe, with regards to the existence of 
``sufficiently large'' regular cardinals that has been a standing assumption along the way (Notation~\ref{notationlarge}). Given a model 
topos of the form $\mathbb{M}=\mathcal{L}_T\mathbf{sPsh}(\mathbf{C})_{\mathrm{inj}}$, for a regular cardinal $\kappa$ to be sufficiently 
large means to be contained in the class $\text{Shl}(\lambda)$ of regular cardinals sharply larger than a specified cardinal $\lambda$ 
associated to the small simplicial category $\mathbf{C}$ -- or to the $\infty$-topos $\mathcal{M}$ that is.\footnote{See \cite{shulmansharppost} for an 
illuminating discussion on directly related issues.} For a universal fibration $\pi_{\kappa}\in\mathbb{M}$ as in Corollary~\ref{coruniverses} 
to be closed under the standard type formers in an appropriate sense (\cite[Section 6]{shulmanuniverses}), the cardinal $\kappa$ has to be 
assumed to be strongly inaccessible.
Thus, if we start with an $\infty$-topos $\mathcal{M}$ and wish to show that its type 
theoretic presentation $\mathbb{M}:=\mathcal{L}_T\mathbf{sPsh}(\mathbf{C})_{\mathrm{inj}}$ exhibits a universal fibration for $\kappa$-small 
fibrations that is closed under all standard type formers, we need an object classifier $V_{\kappa}$ in $\mathcal{M}$ classifying relatively
$\kappa$-compact maps for a strongly inaccessible $\kappa$ \emph{in} $\text{Shl}(\lambda)$. The same holds if we want to show that a given 
cumulative hierarchy of universal fibrations associated to strong inaccessibles $\kappa_i$ in $\mathbb{M}$ yields a corresponding hierarchy 
of object classifiers in $\mathcal{M}$. Thus, the translation of the 
categorical structure together with a universe (or even a cumulative infinite hierarchy of such) 
between Homotopy Type Theory and higher topos theory requires the existence of inaccessibles within any such given class of sharply 
larger cardinals. Fortunately, for strongly inaccessible cardinals $\kappa$ we have $\lambda\ll\kappa$ if 
and only if $\lambda<\kappa$. We hence do not need to make any large cardinal assumptions beyond the existence of the 
inaccessibles themselves.
\end{remark}


\bibliographystyle{amsplain}
\bibliography{BSBib}

\newcommand{\noopsort}[1]{}
\providecommand{\bysame}{\leavevmode\hbox to3em{\hrulefill}\thinspace}
\providecommand{\MR}{\relax\ifhmode\unskip\space\fi MR }
\providecommand{\MRhref}[2]{%
  \href{http://www.ams.org/mathscinet-getitem?mr=#1}{#2}
}
\providecommand{\href}[2]{#2}
\begin{thebibliography}{10}

\bibitem{adamekrosicky}
J.~Ad\'{a}mek and J.~Rosick\'{y}, \emph{Locally presentable and accessible
  categories}, London Mathematical Society Lecture Note Series, vol. 189,
  Cambridge University Press, 1994.

\bibitem{bksimploc}
C.~Barwick and D.~M. Kan, \emph{A characterization of simplicial localization
  functors and a discussion of {DK} equivalences}, Indagationes Mathematicae
  \textbf{23} (2012), no.~1--2, 69--79.

\bibitem{bkrelcat}
\bysame, \emph{Relative categories: Another model for the homotopy theory of
  homotopy theories}, Indagationes Mathematicae \textbf{23} (2012), no.~1--2,
  42--68.

\bibitem{bmezcat}
C.~Berger and I.~Moerdijk, \emph{On an extension of the notion of {R}eedy
  category}, Mathematische Zeitschrift \textbf{269} (2011), no.~3-4, 977 --
  1004.

\bibitem{cisinski}
D.C. Cisinski, \emph{Univalent universes for elegant models of homotopy types},
  \url{http://arxiv.org/abs/1406.0058}, 2014, [Online, accessed 31 May 2014].

\bibitem{duggersmallpres}
D.~Dugger, \emph{Combinatorial model categories have presentations}, Adv. Math.
  \textbf{164} (2001), no.~1, 177--201.

\bibitem{duggerunivhtytheories}
\bysame, \emph{Universal homotopy theories}, Advances in Mathematics
  \textbf{164} (2001), no.~1, 144--176.

\bibitem{dkcalcsimplocs}
W.G. Dwyer and D.~M. Kan, \emph{Calculating simplicial localizations}, Journal
  of Pure and Applied Algebra \textbf{18} (1980), 17--35.

\bibitem{dksimploc}
\bysame, \emph{Simplicial localizations of categories}, Journal of Pure and
  Applied Algebra \textbf{17} (1980), 267--284.

\bibitem{dkequivs}
\bysame, \emph{Equivalences between homotopy theories of diagrams}, Algebraic
  Topology and Algebraic K-theory, Princeton University Press (1987), 180--204.

\bibitem{gambinosmplwghts}
N.~Gambino, \emph{Weighted limits in simplicial homotopy theory}, {J}ournal of
  {P}ure and {A}pplied {A}lgebra \textbf{214} (2010), no.~7, 1193--1199.

\bibitem{gepnerkock}
D.~Gepner and J.~Kock, \emph{Univalence in locally cartesian closed
  infinity-categories}, Forum Mathematicum \textbf{29} (2012), no.~3.

\bibitem{hirschhorn03}
P.S. Hirschhorn, \emph{Model categories and their localizations}, Mathematical
  Surveys and Monographs, no.~99, American Mathematical Society, Providence,
  R.I., 2003.

\bibitem{hovey}
M.~Hovey, \emph{Model categories}, Mathematical Surveys and Monographs,
  vol.~63, American Mathematical Society, 1999.

\bibitem{kellyfinlimits}
G.M. Kelly, \emph{Structures defined by finite limits in the enriched context,
  i}, Cahiers de Topologie et g\'{e}om\'{e}trie diff\'{e}rentielle
  cat\'{e}goriques \textbf{23} (1982), no.~1, 3--42.

\bibitem{kellybook}
\bysame, \emph{Basic concepts of enriched category theory}, Reprints in Theory
  and Applications of Categories, no.~10, 2005, Reprint of the 1982 original
  [Cambridge Univ. Press, Cambridge; MR0651714].

\bibitem{lurieha}
J.~Lurie, \emph{Higher algebra},
  \url{http://www.math.harvard.edu/~lurie/papers/HA.pdf}, Last update September
  2017.

\bibitem{luriehtt}
\bysame, \emph{Higher topos theory}, Annals of Mathematics Studies, no. 170,
  Princeton University Press, 2009.

\bibitem{mocompactobjects}
\bysame, \emph{Compact objects in model categories and
  $(\infty,1)$-categories},
  \url{https://mathoverflow.net/questions/95165/compact-objects-in-model-categories-and-infty-1-categories},
  2012, [Comment from 25./26. April 2012 to M. Shulman's MO post of the given
  title.].

\bibitem{mazelgeernerve}
A.~Mazel-Gee, \emph{The universality of the {R}ezk nerve}, Algebraic \&
  Geometric Topology \textbf{19} (2019), no.~7, 3217–3260.

\bibitem{hott}
The Univalent~Foundations Program, \emph{Homotopy {T}ype {T}heory: {U}nivalent
  {F}oundations of {M}athematics}, \url{http://homotopytypetheory.org/book},
  2013.

\bibitem{rezkhtytps}
C.~Rezk, \emph{Toposes and homotopy toposes (version 0.15)},
  \url{https://www.researchgate.net/publication/255654755_Toposes_and_homotopy_toposes_version_015},
  2010.

\bibitem{riehlcthty}
E.~Riehl, \emph{{C}ategorical {H}omotopy {T}heory}, New Mathematical
  Monographs, vol.~24, Cambridge University Press, 2014.

\bibitem{shulmaninv}
M.~Shulman, \emph{{U}nivalence for inverse diagrams and homotopy canonicity},
  Mathematical Structures in Computer Science \textbf{25} (2015), no.~5,
  1203--1277.

\bibitem{shulmannote}
\bysame, \emph{Presenting $(\infty,1)$-categories with diagrams on relative
  inverse catgories}, Personal communication, 2017.

\bibitem{shulmanuniverses}
\bysame, \emph{{A}ll $(\infty,1)$-toposes have strict univalent universes},
  \url{https://arxiv.org/abs/1904.07004}, 2019, [Online, last revised 26 Apr
  2019].

\bibitem{shulmansharppost}
M.~Shulman, \emph{The myths of presentability and the sharply large filter},
  \url{https://golem.ph.utexas.edu/category/2019/03/the_myths_of_presentability_an.html},
  2019, [Online n-Category Caf\'{e} post on 14 March 2019].

\bibitem{thesis}
R.~Stenzel, \emph{On univalence, {R}ezk completeness and presentable
  quasi-categories}, Ph.D. thesis, University of Leeds, Leeds LS2 9JT, 3 2019.

\bibitem{wolffvloc}
H.~Wolff, \emph{$\mathcal{V}$-categories and $\mathcal{V}$-monads}, Journal of
  Algebra \textbf{24} (1973), 405--–438.

\end{thebibliography}
\Address
\end{document}